\documentclass[a4paper,11pt]{amsart}
\usepackage[utf8]{inputenc}

\usepackage{amsmath}  
\usepackage{amssymb}     
\usepackage{amsthm}  
\usepackage{amsfonts}  
\usepackage{bbm}    
\usepackage{frenchineq}

\usepackage{bm}
\usepackage{hyperref}
\usepackage{accents}
\usepackage{mathrsfs}
\usepackage{mathtools}
\usepackage{fixmath}
\usepackage{fullpage}
\usepackage{paralist}
\usepackage{stmaryrd}
\usepackage{xspace}
\usepackage{bold-extra}

\usepackage{booktabs}
\usepackage{array}

\usepackage{pgf,tikz}
\usetikzlibrary{arrows}
\usetikzlibrary{patterns}
\usetikzlibrary[positioning]
\usetikzlibrary[fit]
\usetikzlibrary{shapes.geometric}
\usetikzlibrary{calc}
\usetikzlibrary{decorations.pathreplacing}

\newcommand{\eg}{e.g.}
\newcommand{\ie}{i.e.}

\newcommand{\Trop}{\mathbb{T}}
\newcommand{\tplus}{\oplus}
\newcommand{\tsum}{\bigoplus}
\newcommand{\tdot}{\odot}

\newcommand{\R}{\mathbb{R}}
\newcommand{\transpose}[1]{{#1}^\top}
\newcommand{\Pcal}{\mathcal{P}}

\newcommand{\Pcalx}{\Pcal'} 
\newcommand{\Qcal}{\mathcal{Q}}
\newcommand{\Rcal}{\mathcal{R}}
\newcommand{\Ccal}{\mathcal{C}}
\newcommand{\Fcal}{\mathcal{F}}
\newcommand{\Ncal}{\mathcal{N}}
\newcommand{\Bcal}{\mathcal{B}}

\newcommand{\Scal}{\mathcal{S}}
\newcommand{\Hcal}{\mathcal{H}}
\newcommand{\Tcal}{\mathcal{T}}
\newcommand{\Ctrop}{\Ccal^{\mathrm{trop}}}
\newcommand{\K}{\mathbb{K}}

\newcommand{\funk}{\delta_{\rm F}}
\newcommand{\hilbert}{d_{\rm H}}
\newcommand{\tsegm}{\mathsf{tsegm}}
\newcommand{\trop}{\mathsf{trop}}
\newcommand{\final}{\underline{\lambda}}
\newcommand{\initial}{\overline{\lambda}}
\newcommand{\mgap}{\bar\mu}
\newcommand{\Z}{\mathbb{Z}}
\newcommand{\LP}{\mathbf{LP}}

\DeclareMathOperator{\val}{\mathrm{val}}
\DeclareMathOperator{\gap}{\mathsf{gap}}

\DeclareMathOperator{\tgap}{\mathsf{tgap}}

\DeclarePairedDelimiter{\abs}{\lvert}{\rvert}
\DeclarePairedDelimiter{\norm}{\lVert}{\rVert}

\newcommand{\arxiv}[1]{\href{http://arxiv.org/abs/#1}{arXiv:#1}}
\newcommand\polymake{\texttt{polymake}\xspace}

\newtheorem{theorem}{Theorem}
\newtheorem{proposition}[theorem]{Proposition}

\newtheorem{corollary}[theorem]{Corollary}
\newtheorem{lemma}[theorem]{Lemma}
\newtheorem{assumption}[theorem]{Assumption}

\newtheorem{introTheorem}{Theorem}

\theoremstyle{remark}
\newtheorem{remark}[theorem]{Remark}

\newcommand{\troppath}{\mathcal{C}^{\text{\rm trop}}}

\newcommand{\cpath}{\mathcal{C}}

\newcommand{\scalar}[2]{\langle#1,#2\rangle}

\newcommand{\tscalar}[2]{\langle #1, #2 \rangle_\Trop}

\DeclareMathOperator*{\argmax}{\arg\max}

\newcommand\Sym[1]{{\operatorname{Sym}}_{#1}}

\newcommand\puiseuxCEX[1]{\textbf{\textsc{LW}}_{#1}}
\newcommand\puiseuxCEXslack[1]{\textbf{\textsc{LW}}_{#1}^{=}}
\newcommand\dualPuiseuxCEXslack[1]{\textbf{\textsc{DualLW}}_{#1}^{=}}
\newcommand\perturbedPuiseuxCEX[2]{\textbf{\textsc{LW}}_{#1}^{#2}}
\newcommand\realCEX[2]{\puiseuxCEX{#1}(#2)}
\newcommand\realCEXslack[2]{\puiseuxCEXslack{#1}(#2)}
\newcommand\dualRealCEXslack[2]{\dualPuiseuxCEXslack{#1}(#2)}

\newcounter{tikzbrace}
\newcommand{\tikzmark}[1]{\tikz[overlay,remember picture] {\node (brace-\thetikzbrace) {};}\stepcounter{tikzbrace}}

\newcommand{\insertbigbrace}[1]{%
\begin{tikzpicture}[remember picture, overlay]
\draw[thick] 
	let \n1 = {\thetikzbrace - 2},
		\n2 = {\thetikzbrace - 1},   
		\p1 = (brace-\n1),
		\p2 = (brace-\n2),
		\n3 = {max(\x1,\x2)},
		\p3 = ($(\n3,\y1) + (0.2,0.4)$),
		\p4 = ($(\n3,\y2) + (0.2,-0.2)$),
		\p5 = (-0.1,0) in 
	(\p3) ++ (\p5) -- (\p3) -- node[right=1ex]{#1} (\p4) -- ++ (\p5);
\end{tikzpicture}}

\newlength{\mytemplen}

\title{Log-barrier interior point methods\\ are not strongly polynomial}

\author{Xavier Allamigeon
\and Pascal Benchimol
\and St\'ephane Gaubert
\and Michael Joswig} 

\address[Xavier Allamigeon, St{\'e}phane Gaubert]{INRIA and CMAP, \'Ecole Polytechnique, CNRS, Université Paris--Saclay, 91128 Palaiseau Cedex, France
\texttt{firstname.lastname@inria.fr}}
\address[Pascal Benchimol]{EDF Lab Paris--Saclay,
7 bd Gaspard Monge,
91120 Palaiseau, France \newline 
\texttt{pascal.benchimol@polytechnique.edu}}
\address[Michael Joswig]{
Institut f{\"u}r Mathematik,
 TU Berlin, 
 Str.\ des 17.\ Juni 136, 10623 Berlin, Germany
 \texttt{joswig@math.tu-berlin.de}
}

\thanks{The first and third authors are partially supported by the PGMO program of EDF and Fondation Math\'ematique Jacques Hadamard.}
\thanks{During this work, P.~Benchimol was affiliated with INRIA Saclay \^Ile-de-France and CMAP, \'Ecole Polytechnique,  CNRS UMR 7641. He was supported a PhD fellowship of DGA and \'Ecole Polytechnique.}
\thanks{M.~Joswig is partially supported by Einstein Foundation Berlin, DFG within the Priority Program 1489 and by a CNRS INSMI visiting professorship at CMAP, \'Ecole Polytechnique, UMR 7641 and IMJ, Universit\'e Pierre et Marie Curie, UMR 7586.}

\date{\today}

\begin{document}

\begin{abstract}
We prove that primal-dual log-barrier interior point methods 
are not strongly polynomial, by constructing a family of linear programs with
 $3r+1$ inequalities in dimension $2r$ for which the number of iterations performed 
is in $\Omega(2^r)$.
The total curvature
of the central path of these linear programs
is also exponential in $r$,
disproving a continuous analogue of the Hirsch conjecture proposed by Deza, Terlaky and Zinchenko.  
Our method is to tropicalize the central
  path in linear programming. The tropical central path is the piecewise-linear limit of the central paths of parameterized families of classical
  linear programs viewed through logarithmic glasses. This allows us
to provide combinatorial lower bounds for the number of iterations and the total curvature, in a general setting.

\end{abstract}

\maketitle

\section{Introduction}\label{sec:introduction}

\noindent
An open question in computational optimization, known as Smale's 9th problem~\cite{smale}, asks whether linear programming can be solved with a strongly polynomial
algorithm. This requires the 
algorithm to be polynomial time in the bit model
(meaning that the 
execution time is polynomial in the number of bits of the input)
\emph{and} the number of arithmetic operations to be bounded by a polynomial
in the number of numerical entries of the input, uniformly in their bit length.

It is instructive to consider interior point methods in view of this question. Since Karmarkar's seminal work~\cite{karmarkar1984new}, the latter
have become indispensable in mathematical optimization.
Path-following interior point methods are driven to an optimal solution along a trajectory called the \emph{central path}.
The best known upper bound on the number of iterations performed by path-following interior point methods for linear programming is $O(\sqrt n L)$, where $n$ is the number of variables and $L$ is the total bit size of all coefficients. Hence, they are polynomial in the bit model (every iteration can be done in strongly polynomial time). It is tempting to ask whether a suitable interior point method could lead to a strongly polynomial algorithm in linear programming. 
In other words, this raises the question of bounding the number of iterations by a polynomial depending only on the number of variables and constraints.

Early on, Bayer and Lagarias recognized that the central path is ``a fundamental mathematical object underlying Karmarkar's algorithm and that the good convergence properties of Karmarkar's algorithm arise from good geometric properties of the set of trajectories''~\cite[p.~500]{BayerLagarias89a}. 
Such considerations led Dedieu and Shub to consider the total curvature as an informal complexity measure of the central path.
Intuitively, a central path with a small total curvature should be easier to approximate by linear segments.
They
conjectured that the total curvature of the central path is linearly bounded in the dimension of the ambient space~\cite{dedieu2005newton}.
Subsequently, Dedieu, Malajovich and Shub showed that this property is valid in an average sense~\cite{dedieu2005curvature}. 
However, Deza, Terlaky and Zinchenko provided a counterexample by constructing a redundant Klee-Minty cube~\cite{DTZ08}.
This led them to propose a revised conjecture, the ``continuous analogue of the Hirsch conjecture'', which says that the total curvature of the central path is linearly bounded in the number of constraints. 

In this paper, we disprove the conjecture of Deza, Terlaky and Zinchenko,
by constructing a family of linear programs for which the total curvature is exponential
in the number of constraints. Moreover, we show that for the same family of linear
programs, 
a significant class of polynomial time interior point methods, namely the primal-dual path-following methods with respect to a log-barrier function, are not strongly polynomial.

More precisely, given $r\geq 1$, we consider the linear program
\begin{equation}\label{eq:counterexample}
\begin{array}{r@{\quad}l}
\text{minimize} & x_1 \\[\jot]
\text{subject to} & 
x_1 \leq t^2 \\[\jot]
& x_2 \leq t \\[\jot]
& x_{2j+1} \leq t \, x_{2j-1} \, , \; x_{2j+1} \leq t \, x_{2j} \tikzmark{} \\[\jot]
& x_{2j+2} \leq t^{1-1/2^j} (x_{2j-1} + x_{2j}) \tikzmark{} \\[\jot]
& x_{2r-1} \geq 0 \, , \; x_{2r} \geq 0
\end{array}\tag*{$\realCEX{r}{t}$}
\insertbigbrace{$1 \leq j < r$}
\end{equation}
which depends on a parameter $t > 0$.
The linear program~$\realCEX{r}{t}$ has $2r$ variables and $3r+1$ constraints.
The notation $\realCEX{r}{t}$ for these linear programs refers to the ``long and winding'' nature of their central paths.
More precisely, our first main result is the following.

\begin{introTheorem}[see Theorem~\ref{th:curvature}] \label{thm:curvature:intro}
  The total curvature of the central path of the linear programs $\realCEX{r}{t}$ is exponential in $r$, provided that $t>0$ is sufficiently large.
\end{introTheorem}

Our second main result provides an exponential lower bound for the number of iterations of a large class of path-following interior point methods. We only require these methods to stay in the so-called ``wide'' neighborhood of the central path; see Section~\ref{sec:classical_central_path} for the definition. Remarkable examples of such methods include short or long-step methods, like the ones of Kojima, Mizuno and Yoshise~\cite{kojima_short_step,kojima_long_step} and Monteiro and Adler~\cite{AdlerMonteiro}, as well as predictor-corrector methods like the ones of Mizuno, Todd and Ye~\cite{MizunoToddYe} and Vavasis and Ye~\cite{VavasisYe}.

\begin{introTheorem}[see Corollary~\ref{cor:exp_lower_bound}]\label{thm:complexity:intro}
The number of iterations of any primal-dual path-following interior point algorithm with a log-barrier function which iterates in the wide neighborhood of the central path is exponential in $r$ on the linear programs $\realCEX{r}{t}$, provided that $t>0$ is sufficiently large.
\end{introTheorem}

The proofs of these theorems rely on tropical geometry.
The latter can be seen as the (algebraic) geometry on the tropical (max-plus) semifield $(\Trop, \tplus, \tdot)$ where the set $\Trop = \R \cup \{ - \infty \}$ is endowed with the operations $a \tplus b = \max(a,b)$ and $a \tdot b = a + b$.
A tropical variety can be obtained as the limit at infinity of a sequence of classical algebraic varieties depending on one real parameter $t$ and drawn on logarithmic paper, with $t$ as the logarithmic base. 
This process is known as Maslov's~\emph{dequantization}~\cite{litvinov2007maslov}, or Viro's method~\cite{viro2001dequantization}.
It can be traced back to the work of Bergman~\cite{bergman1971logarithmic}.
In a way, dequantization yields a piecewise linear image of classical algebraic geometry. 
Tropical geometry has a strong combinatorial flavor, and yet it retains a lot of information about the classical objects~\cite{itenberg2009tropical,TropicalBook}.
 
The tropical semifield can also be thought of as the image of a non-Archimedean field under its valuation map. This is the approach we adopt here, by considering
 $\realCEX{r}{t}$ as a linear program over a real closed non-Archimedean field of Puiseux series in the parameter $t$. Then, the tropical central path is defined
as the image by the valuation of the central path over this field.
We first give an explicit geometric characterization of the tropical central path, as a tropical analogue of the barycenter of a sublevel set of the feasible set induced by the duality gap; see Section~\ref{subsec:geometric_characterization}. Interestingly, it turns out that the tropical central path does not depend on the external representation of the feasible set. This is in stark contrast with the classical case; see~\cite{DTZ08} for an example. We study the convergence properties of the classical central path to the tropical one in Section~\ref{subsec:uniform}.

We then show that, when $t$ is specialized to a suitably large real value,
the total curvature of the central path of the linear program $\realCEX{r}{t}$ is bounded below by a combinatorial curvature (the tropical total curvature)
depending only on the image of the central path by the nonarchimedean valuation (Section~\ref{sec:curvature}).
The linear programs $\realCEX{r}{t}$ have an inductive construction, leading to a tropical central path with a self-similar pattern, resulting
in an exponential
number of sharp turns as $r$ tends to infinity. In this way,
we obtain the exponential 
bound for the total curvature (Theorem~\ref{thm:curvature:intro}). 

A further refinement of the tropical analysis shows that the number of iterations performed by interior point methods is bounded from below by the number of tropical segments constituting the tropical central path; see Section~\ref{sec-cb}. For the family of linear programs $\realCEX{r}{t}$,
we show that the number of such segments
is necessarily exponential, leading to the exponential
lower bound for the number of iterations of interior
point methods 
(Theorem~\ref{thm:complexity:intro}). We provide an explicit lower bound for the value of the parameter $t$. It is doubly exponential in $r$, implying that
 the bitlength of $t$ is exponential in $r$, which is consistent with the polynomial time character of interior point methods in the bit model.

\subsection*{Related Work.}
The redundant Klee-Minty cube of~\cite{DTZ08} and the ``snake'' in~\cite{deza2008polytopes} are instances which show that the total curvature of the central path can be in $\Omega(m)$ for a polytope described by $m$ inequalities.  
Gilbert, Gonzaga and Karas~\cite{CharlesGilbert2004} also
exhibited ill-behaved central paths.  They showed that the central path can have a ``zig-zag'' shape with infinitely
many turns, on a problem defined in $\R^2$ by non-linear but convex functions.  

The central path has been studied by Dedieu, Malajovich and Shub~\cite{dedieu2005curvature} via the multihomogeneous B\'ezout
Theorem and by De Loera, Sturmfels and Vinzant~\cite{de2010central} using matroid theory.  These two papers provide an
upper bound of $O(n)$ on the total curvature averaged over all regions of an arrangement of hyperplanes in
dimension~$n$.  

In terms of iteration-complexity of
interior point methods, several worst-case results have been
proposed~\cite{anstreicher1991performance,kaliski1991convergence,ji1994complexity,powell1993number,todd1996lower,Bertsimas1997}.
In particular, Stoer and Zhao~\cite{zhao1993estimating} showed the iteration-complexity of a certain class of
path-following methods is governed by an integral along the central path. This quantity, called \emph{Sonnevend's
  curvature}, was introduced in~\cite{sonnevend1991complexity}. The tight relationship between the total Sonnevend
curvature and the iteration complexity of interior points methods have been extended to semidefinite and symmetric cone
programs~\cite{kakihara2013information}. 
Note that Sonnevend's curvature is different from the \emph{geometric} curvature we study in this paper. To the best of
our knowledge, there is no explicit relation between the geometric curvature and the iteration-complexity of
interior point methods. We circumvent this difficulty here by showing
directly that the geometric curvature {\em and} the number of iterations
are exponential for the family $\realCEX{r}{t}$.

The present work relies on the considerations of amoebas (images by the valuation) of algebraic and semi-algebraic sets, see~\cite{kapranov,DevelinYu07,richter2005first,alessandrini2013} for background.
Another ingredient is a construction of Bezem, Nieuwenhuis and
Rodr{\'{\i}}guez-Carbonell~\cite{BezemNieuwenhuisRodriguez08}. Their goal was to show that an algorithm of Butkovi\v{c}
and Zimmermann~\cite{Butkovic2006} has exponential running time.
We arrive at our family of linear programs by lifting a variant of this construction to Puiseux series.

Finally, let us point out that the first main result of this article,
the exponential bound for the total curvature of the central path,
initially appeared in our preprint~\cite{longandwinding}.
We next discuss the main differences with the present paper.
In the original preprint, we exploited different tools: to characterize the tropical central path, we used methods from model theory, employing o-minimal structures and Hardy fields, in the spirit of Alessandrini~\cite{alessandrini2013}. In the present revision,
we provide a more elementary proof, avoiding the use of model theory. The original model theory approach, however, keeps the advantage of greater generality.
We expect that this will allow to extend some of the present results to other kinds of barrier functions.
In addition, we have added an explicit lower bound for the number of iterations, our second main result here.
We thank the colleagues who commented on~\cite{longandwinding}, in particular A.~Deza, T.~Terlaky and Y.~Zinchenko.

\section{The Primal-Dual Central Path and Its Neighborhood}
\label{sec:classical_central_path}
\noindent
In this section, we recall the definition of the central path and introduce the notions related to path-following interior point methods which we will use in the rest of the paper. 

In what follows, we consider a linear program of the form
\begin{equation}\label{LP:primal}
  \text{minimize} \quad \scalar{c}{x} \quad \text{subject to} \quad A x + w = b \, , \;  (x, w) \in \R_+^{n+m} \tag*{$\text{LP}(A,b,c)$}
\end{equation}
in which the slack variables $w$ are explicit. Here and below,
$A$ is a $m\times n$ matrix, $b\in \R^m$, $c\in \R^n$,
we denote by $\scalar{\cdot}{\cdot}$ the standard scalar product, and by $\R_+$ the set of non-negative reals. The dual linear program takes a form which is very similar:
\begin{equation}\label{LP:dual}
  \text{maximize} \quad \scalar{b}{y} \quad \text{subject to} \quad s - \transpose{A} y = c \, , \;  (s, y) \in \R_+^{n+m} \tag*{$\text{DualLP}(A,b,c)$} \enspace,
\end{equation}
where $\cdot^\top$ denotes the transposition.
For the sake of brevity, we set $N \coloneqq n + m$ to represent the total number of variables in both linear programs.
Let $\Fcal^\circ$ be the set of strictly feasible primal-dual elements, \ie,
\[
\Fcal^\circ \coloneqq \bigl\{ z = (x, w, s, y) > 0 \colon A x + w = b \, , \; s -\transpose{A} y  = c
\bigr\}	\, .
\]
which we assume to be nonempty.
In this situation, for any given $\mu>0$, the system of equations and inequalities
\begin{equation} \label{eq:classical_central_path}
\begin{aligned}
  A x + w & = b \\
s-\transpose{A} y & =c \\
\begin{pmatrix}
x s \\
w y    
\end{pmatrix}
& = \mu e \\
x, w, y, s & > 0
\end{aligned} 
\end{equation}
is known to have a unique solution $(x^\mu, w^\mu, s^\mu, y^\mu) \in \R^{2N}$; here $e$ stands for the all-$1$-vector in $\R^N$; further, $x s$ and $w y$ denote the Hadamard products of $x$ by $s$ and $w$ by $y$, respectively. The \emph{central path} of the dual pair~\ref{LP:primal} and~\ref{LP:dual} of linear programs is defined as the function which maps $\mu > 0$ to the point $(x^\mu, w^\mu, s^\mu, y^\mu)$. The latter shall be referred to as the \emph{point of the central path with parameter} $\mu$.
The equality constraints in \eqref{eq:classical_central_path} define a real algebraic curve, the \emph{central curve} of the dual pair of linear programs, which has been studied in \cite{BayerLagarias89a} and \cite{de2010central}.
The central curve is the Zariski closure of the central path.

The \emph{primal} and \emph{dual central paths} are defined as the projections of the central path onto the $(x, w)$-
and $(s,y)$-coordinates, respectively. Equivalently, given $\mu > 0$, the points $(x^\mu, w^\mu)$ and $(s^\mu, y^\mu)$
on the primal and dual central paths can be defined as the unique optimal solutions of the following pair of
\emph{logarithmic barrier} problems:
\[
\begin{array}{r@{\quad}l}
\text{minimize} & \scalar{c}{x} - \mu \Bigl(\sum_{j = 1}^n \log(x_j) + \sum_{i = 1}^m \log(w_i)\Bigr) \\[\jot]
\text{subject to} & A x + w = b \, , \, x > 0 \, , \, w > 0 \, ,
\end{array}
\]
and:
\[
\begin{array}{r@{\quad}l}
\text{maximize} & \mu \Bigl(\sum_{j = 1}^n \log(s_j) + \sum_{i = 1}^m \log(y_i)\Bigr) - \scalar{b}{y}\\[\jot]
\text{subject to} & s - \transpose{A} y = c \, , \, s > 0 \, , \, y > 0 \, .
\end{array}
\]
The uniqueness of the optimal solutions follows from the fact that the objective functions are strictly convex and concave, respectively.
The equivalence to~\eqref{eq:classical_central_path} results from the optimality conditions of the logarithmic barrier problems.
The main property of the central path is that the sequences $(x^\mu, w^\mu)$ and $(s^\mu, y^\mu)$ converge to optimal solutions $(x^*, w^*)$ and $(s^*, y^*)$ of the linear programs~\ref{LP:primal} and~\ref{LP:dual}, when $\mu$ tends to $0$.

The \emph{duality measure} $\bar{\mu}(z)$ of an arbitrary point $z = (x, w, s, y)\in\R_+^{2N}$ is defined by
\begin{equation}\label{eq:duality_measure}
  \bar{\mu}(z) \coloneqq \frac{1}{N}\bigl(\scalar{x}{s} + \scalar{w}{y}\bigr) \, .
\end{equation}
With this notation, observe that the point $z$ belongs to the central path if and only if we have
\begin{equation}\label{eq:centralpath}
  \begin{pmatrix}
    x s \\
    w y	
  \end{pmatrix}
  = \bar{\mu}(z) e \, .
\end{equation}
In other words, the difference $\begin{psmallmatrix}
x s \\
w y	
\end{psmallmatrix} - \bar{\mu}(z) e$ 
indicates
how far the point $z = (x,w,s,y)$ is
from
the central path.
This leads to introducing the neighborhood
\begin{equation}\label{eq:neighborhood}
\Ncal_\theta \coloneqq \Bigl\{ z \in {\Fcal}^\circ\colon
\Bigl\| 
\begin{pmatrix}
x s \\
w y 
\end{pmatrix} - 
\bar{\mu}(z) e \Bigr\| \leq \theta \bar{\mu}(z) \Bigr\} \, ,
\end{equation}
of the central path by bounding some norm of the deviation in terms of a precision parameter $0 < \theta < 1$.
Clearly, this neighborhood depends on the choice of the norm $\| \cdot \|$.
In the context of interior point methods common choices include the $\ell_2$- or $\ell_\infty$-norms.
However, here we focus on the \emph{wide neighborhood}
\begin{equation}\label{eq:wide_neighborhood}
  \Ncal^{-\infty}_\theta \coloneqq \Bigl\{ z \in {\Fcal}^\circ \colon 
  \begin{pmatrix}
    x s \\
    w y 
  \end{pmatrix} \geq (1 - \theta) \bar{\mu}(z) e \Bigr\} \, .
\end{equation}
This arises from replacing $\| \cdot \|$ in~\eqref{eq:neighborhood} by the \emph{one-sided $\ell_\infty$-norm}. The latter is
the map sending a vector $v$ to $\max(0,\max_i (-v_i))$. 
This is a weak norm in the sense of~\cite{PT14}, it 
is positively homogeneous and subadditive,
but it vanishes on some non-zero vectors.

Our first observation is that the map $\mgap$ commutes with affine combinations.
\begin{proposition}\label{prop:mgap_affine}
  Let $z = (x,w,s,y)$ and $z' = (x',w',s',y')$ be two points in $\Fcal$.
  Then, for all $\alpha\in\R$,
we have
  \[
  \mgap((1-\alpha) z + \alpha z') = (1-\alpha) \mgap(z) + \alpha \mgap(z') \, .
  \]
\end{proposition}
\begin{proof}
We write $\Delta z \coloneqq z'- z$ and similarly for the components, \ie, $\Delta z = (\Delta x, \Delta w, \Delta s, \Delta y)$. 
Since $z, z' \in \Fcal$, we have $A \Delta x + \Delta w = 0$ and $\Delta s - \transpose{A} \Delta y = 0$.
Employing these equalities it can be verified that $\scalar{\Delta x}{\Delta s} + \scalar{\Delta w}{\Delta y} = 0$.
Therefore, the function $\alpha \mapsto \mgap(z + \alpha \Delta z)$ is
affine, which completes the proof. 
\end{proof}
Interior point methods follow the central path by computing a sequence of points in a prescribed neighborhood $\Ncal_\theta$ of the central path, in such a way that the duality measure decreases.
Here, we do not precisely specify $\Ncal_\theta$, but we only assume that it arises from the choice of some 
weak norm $\|\cdot\|$. The basic step of the algorithm can be summarized as follows. At iteration $k$, given a current point $z^k = (x^k, w^k, s^k, y^k) \in \Ncal_\theta$ with duality measure $\mu^k = \mgap(z^k)$, and a positive parameter $\mu < \mu^k$, the algorithm aims at solving the system \eqref{eq:classical_central_path} up to a small error in order to get an approximation of the point of the central path with parameter $\mu$. To this end, it starts from the point $z^k$ and exploits the Newton direction $\Delta z = (\Delta x, \Delta w, \Delta s, \Delta y)$, which satisfies
\begin{equation}
\begin{aligned}
A \Delta x + \Delta w & = 0 \\
\Delta s - \transpose{A}\Delta y & = 0 \\
\begin{pmatrix}
x^k \Delta s \\
w^k \Delta y
\end{pmatrix}
+ 
\begin{pmatrix}
\Delta x \, s^k \\
\Delta w \, y^k
\end{pmatrix}
& = 
\mu e  - \begin{pmatrix}
x^k s^k \\
w^k y^k 
\end{pmatrix} \, .
\end{aligned}\label{eq:newton}
\end{equation}
Then, the algorithm follows the direction $\Delta z$ and iterates to a point of the form
\[
z(\alpha) \coloneqq z^k + \alpha \Delta z \, ,
\]
where $0 < \alpha \leq 1$. The correctness and the convergence of the approach are based on the conditions that, firstly, the point $z^{k+1} \coloneqq z(\alpha)$ still belongs to the neighborhood $\Ncal_\theta$ and that, secondly, the ratio of the new value $\mu^{k+1} \coloneqq \bar \mu (z(\alpha))$ of the duality measure with $\mu^k$ is sufficiently small. The following lemma shows that, in fact, the whole line segment between $z^k$ and $z(\alpha)$ is contained in $\Ncal_\theta$.

\begin{lemma}\label{lemma:newton}
If $z^k$ and $z(\alpha)$ are contained in $\Ncal_\theta$ then $z(\beta) \in \Ncal_\theta$ for all $\beta \in [0,\alpha]$.
\end{lemma}

\begin{proof}
The point $z(\beta)$ lies in $\Fcal^\circ$ since the latter set is convex. We use the notation $z(\beta) = (x(\beta), w(\beta), s(\beta), y(\beta))$. Using the last equality in~\eqref{eq:newton}, we can write
\[
\begin{psmallmatrix}
x(\beta) s(\beta) \\
w(\beta) y(\beta)
\end{psmallmatrix}
= \beta \mu e + (1 - \beta) 
\begin{psmallmatrix} 
x^k s^k \\
w^k y^k
\end{psmallmatrix} + \beta^2 
\begin{psmallmatrix}
\Delta x \Delta s \\
\Delta w \Delta y	
\end{psmallmatrix} \, .
\]
The first two equalities in~\eqref{eq:newton} entail $\scalar{\Delta x}{\Delta s} + \scalar{\Delta w}{\Delta y} = 0$. Exploiting the last equality, we
get $\bar{\mu}(z(1))=\mu$. Then, it follows
from Proposition~\ref{prop:mgap_affine} that $\bar{\mu}(z(\beta))=(1-\beta)\mu^k
+ \beta \mu$. 
We deduce that
\[
\begin{psmallmatrix}
x(\beta) s(\beta) \\
w(\beta) y(\beta)
\end{psmallmatrix} - \bar \mu(z(\beta)) e = 
(1 - \beta) 
\Bigl[\begin{psmallmatrix} 
x^k s^k \\
w^k y^k
\end{psmallmatrix}
- \mu^k e \Bigr]
+ \beta^2 \begin{psmallmatrix}
\Delta x \Delta s \\
\Delta w \Delta y	
\end{psmallmatrix} \, .
\]
The same relation holds when $\beta = \alpha$. In this way, we eliminate the term $\begin{psmallmatrix}
\Delta x \Delta s \\
\Delta w \Delta y	
\end{psmallmatrix}$ to write
\[
\begin{psmallmatrix}
x(\beta) s(\beta) \\
w(\beta) y(\beta)
\end{psmallmatrix} - \bar \mu(z(\beta)) e =
\bigl((1 - \beta) - \frac{\beta^2}{\alpha^2} (1-\alpha)\bigr)
\Bigl[\begin{psmallmatrix} 
x^k s^k \\
w^k y^k
\end{psmallmatrix}
- \mu^k e \Bigr]
+ \frac{\beta^2}{\alpha^2} \Bigl[\begin{psmallmatrix}
x(\alpha) s(\alpha) \\
w(\alpha) y(\alpha)
\end{psmallmatrix} - \bar \mu(z(\alpha)) e  \Bigr] \, .
\]
Since $\beta \leq \alpha \leq 1$, the term $(1 - \beta) - \frac{\beta^2}{\alpha^2} (1-\alpha)$ is non-negative. Using the fact that $z^k$ and $z(\alpha)$ belong to $\Ncal_\theta$, and the subadditivity and positive homogeneity of the weak norm $\|\cdot\|$, we deduce that
\begin{align*}
\Bigl\|
\begin{psmallmatrix}
x(\beta) s(\beta) \\
w(\beta) y(\beta)
\end{psmallmatrix} - \bar \mu(z(\beta)) e
\Bigr\|
& \leq 
\bigl((1 - \beta) - \frac{\beta^2}{\alpha^2} (1-\alpha)\bigr) \theta \mu^k + 
\frac{\beta^2}{\alpha^2} \theta \bar \mu(z(\alpha)) \\
& = \theta \bigl((1-\beta) \mu^k + \frac{\beta^2}{\alpha} \mu\bigr) \\
& \leq \theta \bigl((1-\beta) \mu^k + \beta \mu\bigr) = \theta \bar \mu(z(\beta)) \quad \text{as} \ \beta \leq \alpha \, . \qedhere
\end{align*}
\end{proof}
The implementation of the basic iteration step which we have previously described varies from one interior point method to another. In particular, there exists several strategies for the choice of the neighborhood, the parameter $\mu$ with respect to the current value of the duality measure $\mu^k$, and the step length $\alpha$, in order to achieve a polynomial-time complexity. Let us describe in more detail the main ones. Considering the large variety of existing path-following interior point methods in the literature, we stick to the classification of~\cite[Chapter~5]{Wright}, and refer to it for a complete account on the topic.

Short-step interior point methods, like~\cite{kojima_short_step,AdlerMonteiro}, use an $\ell_2$-neighborhood of prescribed size $\theta$, and set $\mu$ to $\sigma \mu^k$ where $\sigma < 1$ is constant throughout the method (chosen in a careful way to ensure the convergence), and $\alpha$ to $1$. In contrast, long-step interior point methods, such as~\cite{kojima_long_step}, exploit the wider neighborhood $\Ncal_\theta^{-\infty}$, allow more freedom for the choice of $\mu$ at every iteration ($\mu$ is set to $\sigma \mu^k$ where $\sigma < 1$ is chosen in prescribed interval $[\sigma^{\min}, \sigma^{\max}]$), and take $\alpha \in [0,1]$ as large as possible to ensure that $z(\alpha) \in \Ncal_\theta^{-\infty}$. Another important class of methods, the so-called \emph{predictor-corrector} ones, make use of two nested $\ell_2$-neighborhoods $\Ncal_{\theta'}$ and $\Ncal_{\theta}$ ($\theta' < \theta$), and alternate between predictor and corrector steps. In the former, $\mu$ is optimistically set to $0$ (the duality measure of optimal solutions), while $\alpha$ is chosen as the largest value in $[0,1]$ such that $z(\alpha) \in \Ncal_{\theta}$. The next corrector step aims at ``centering'' the trajectory by doing one Newton step in the direction of the point of the central path with parameter $\mu^{k+1} = \bar \mu(z(\alpha))$. This means that the duality measure is kept to $\mu^{k+1}$, and the step length is set $1$. A careful choice of $\theta$ depending on $\theta'$ ensures that we obtain in this way a point in the narrower neighborhood $\Ncal_{\theta'}$. 

The predictor-corrector scheme, initially introduced in~\cite{MizunoToddYe}, has inspired several works. Let us mention the one of Vavasis and Ye~\cite{VavasisYe}, who made a step towards a strongly polynomial complexity by arriving at an iteration complexity upper bound depending on the matrix $A$ only. Their technique has been later refined into more practical algorithms, see~\cite{MegiddoMizunoTsuchiya,MonteiroTsuchiya,KitaharaTsuchiya}. The difference of these methods with the original predictor-corrector one is that they sometimes exploit another direction than the Newton one, called the \emph{layered least squares direction}. However, in such iterations, the step length $\alpha$ is always chosen so that for all $0 \leq \beta \leq \alpha$, the point $z(\beta)$ lies in the neighborhood $\Ncal_{\theta}$ (see~\cite[Theorem~9]{VavasisYe}).

In consequence of Lemma~\ref{lemma:newton} and the previous discussion, the aforementioned interior point methods all share the property that they describe a piecewise linear trajectory entirely included in a certain neighborhood $\Ncal_\theta^{-\infty}$ of the central path, where $\theta$ is a prescribed value.\footnote{We point out that the $\ell_2$-neighborhood with size $\theta$ is obviously contained in $\Ncal_\theta^{-\infty}$.} We stress that this property is the only assumption made in our complexity result, Theorem~\ref{thm:complexity:intro}, on interior point methods. More formally, this trajectory is a \emph{polygonal curve} in $\R^{2N}$, \ie, a union of finitely many segments $[z^0, z^1], [z^1, z^2], \dots, [z^{p-1}, z^p]$. Since polygonal curves play an important role in the paper, we  introduce some terminology.  We say that a polygonal curve is \emph{supported} by the vectors $v^1, \dots, v^p$ when the latter correspond to the direction vectors of the successive segments $[z^0, z^1], [z^1, z^2], \dots, [z^{p-1}, z^p]$. In the case where we equip the curve with an orientation, we assume that the direction vectors are oriented consistently.

\section{Ingredients From Tropical Geometry}
\noindent
Tropical geometry provides a combinatorial approach to studying algebraic varieties defined over a field with a non-Archimedean valuation.
To deal with optimization issues, we need some valued field which is ordered.
We restrict our attention to one such field, which is particularly convenient for our application, to keep our exposition elementary.
Our field of choice, which we denote as $\K$, are the absolutely convergent generalized real Puiseux series.
Here `generalized' means that we allow arbitrary real numbers as exponents as in \cite{markwig2007field}.
Note that the ordinary Puiseux series have value group $\mathbb{Q}$, leading to restrictions which are artificial from a tropical perspective.
In some sense, $\K$ is the ``simplest'' real closed valued field for which we can obtain our results.

\subsection{Fields of real Puiseux series and Puiseux polyhedra}
\label{subsec:fields}
The field $\K$ of \emph{absolutely convergent generalized real Puiseux series} consists of elements of the form
\begin{equation}
\bm f = \sum_{\alpha \in \R} a_\alpha t^\alpha \, ,
\label{e-series}
\end{equation}
where $a_\alpha\in \R$ for all $\alpha$, and such that: %
\begin{inparaenum}[(i)]
\item the support $\{\alpha \in \R\colon a_\alpha \neq 0\}$ is either finite or has $-\infty$ as the only accumulation point; 
\item there exists $\rho > 0$ such that the series absolutely converges for all $t > \rho$.
\end{inparaenum}
Note that the null series is obtained by taking an empty support. When $\bm f \neq 0$, the first requirement ensures that the support has a greatest element $\alpha_0 \in \R$. We say that the element $\bm f$ is \emph{positive} when the associated coefficient $a_{\alpha_0}$ is positive. This extends to a total ordering of $\K$, defined by $\bm f \leq \bm g$ if $\bm g - \bm f$ is the null series or positive. Equivalently, the relation $\bm f \leq \bm g$ holds if and only if $\bm f(t) \leq \bm g(t)$ for all sufficiently large $t$. We write $\K_+$ for the set of non-negative elements of $\K$. 

The \emph{valuation} map $\val \colon \K \to \R \cup \{-\infty\}$ is given as follows. For $\bm f \in \K$ the valuation $\val(\bm f)$ is defined as the greatest element $\alpha_0$ of the support of $\bm f$ if $\bm f \neq 0$, and $-\infty$ otherwise. Denoting by $\log_t(\cdot) \coloneqq \frac{\log(\cdot)}{\log t}$ the logarithm with respect to the base $t > 0$ we have
\[
\val(\bm f) = \lim_{ t \to +\infty } \log_t |\bm f(t)| \, ,
\]
with the convention $\log_t 0 = -\infty$. Observe that, for all $\bm f, \bm g \in \K$, this yields
\begin{align}\label{e-morphism}
 \val (\bm f + \bm g) \leq \max ( \val( \bm f), \val( \bm g)) \quad \text{and} \quad \val (\bm f \bm g ) = \val (\bm f) + \val (\bm g) \, .
\end{align}
The inequality for the valuation of the sum turns into an equality when the leading terms in the series $\bm f$ and $\bm g$ do not cancel. In particular, this is the case when $\bm f$ and $\bm g$ belong to $\K_+$.

We point out that $\K$ actually agrees with the field of generalized Dirichlet series originally considered by Hardy and Riesz~\cite{hardy}.
This was already used in the tropical setting in~\cite{ABG96}.
Classical Dirichlet series can be written as $\sum_k a_k k^s$, and these are obtained from~\eqref{e-series} by substituting $t=\exp(s)$ and $\alpha_k =\log k$.
It follows from results of van den Dries and Speissegger~\cite{Dries1998} that the field $\K$ is real closed.
The interest in such fields comes from \emph{Tarski's Principle}, which says that every real closed field has the same first-order properties as the reals.

As a consequence of the previous fact, we can define polyhedra over Puiseux series as usual. In more details, given $d \geq 1$, a \emph{(Puiseux) polyhedron} is a set of the form 
\begin{equation}\label{eq:puiseux_polyhedron}
\bm \Pcal = \{ \bm x \in \K^d \colon \bm A \bm x \leq \bm b \} \, ,
\end{equation}
where $\bm A \in \K^{p \times d}$, $\bm b \in \K^p$ (with $p \geq 0$), and $\leq$ stands for the partial order over $\K^p$. By Tarski's principle, Puiseux polyhedra have the same (first-order) properties as their analogs over $\R$. In particular, the Minkowski--Weyl theorem applies, so that every Puiseux polyhedron admits an internal representation by means of a finite set of points and rays in $\K^d$. 

Using a field of convergent series allows us to think of Puiseux polyhedra as parametric families of ordinary polyhedra. Indeed, to any Puiseux polyhedron $\bm \Pcal$ of the form~\eqref{eq:puiseux_polyhedron}, we associate the family of polyhedra $\bm\Pcal(t)\subset \R^d$, defined for $t$ large enough,
\[ 
\bm\Pcal(t) := \{x\in \R^d\colon \bm{A}(t) x\leq \bm{b}(t)\} \, .
\]
The next proposition implies in particular that the family
of polyhedra $\bm\Pcal(t)$ is independent of the choice
of the external representation of $\bm\Pcal$.
\begin{proposition}
\label{prop-bygenerators}
Suppose that 
$\bm\Pcal$ is the Minkowski sum of the convex hull of vectors $\bm u^1,\dots,\bm u^q\in \K^d$ and of the convex cone generated by vectors $\bm v^1,\dots, \bm v^r\in \K^d$ (here, the notions of convex hull and of convex cone are understood
over $\K$). Then, for $t$ large enough, $\bm\Pcal(t)$ is 
the Minkowski sum of the convex hull of vectors $\bm u^1(t),\dots,\bm u^q(t)\in \R^d$ and of the convex cone generated by vectors $\bm v^1(t),\dots, \bm v^r(t)\in \R^d$ (the notions of convex hull and of convex cone are now understood
over $\R$). 
\end{proposition} 
\begin{proof}
Let $\bm\Qcal(t)$
denote the Minkowski sum of the convex hull of vectors $\bm u^1(t),\dots,\bm u^q(t)$ and of the convex cone generated by vectors $\bm v^1(t),\dots, \bm v^r(t)$. Since $\bm\Pcal$ contains $\bm u^1,\dots,\bm u^q$
together with the rays $\K_+ \bm v^1,\dots,\K_+ \bm v^r$,
we have $\bm A \bm u^i\leq \bm b$ and $\bm A\bm v^j \leq 0$
for all $i\in [q]$ and $j\in[r]$. It follows
that $\bm A(t) \bm u^i(t)\leq \bm b(t)$ and
$\bm A(t)\bm v^j(t) \leq 0$ holds 
for all $i\in [q]$ and $j\in[r]$ and
for $t$ large enough. Hence, $\bm\Pcal(t)\supset \bm\Qcal(t)$
for $t$ large enough.

Let us now consider an extreme point $\bm u$ of $\bm \Pcal$.
Then, a characterization of the extreme points of a polyhedron
shows that the collection of gradients of the constraints
$\bm A_k \bm x \leq \bm b_k$, $k\in [p]$ 
which are active
at point $\bm x=\bm u$ constitutes
a family of full rank. This property can be expressed
in the first order theory of $\K$. It follows that,
for $t$ large enough, the same property holds
for the collection of the 
gradients of the constraints
$\bm A_k(t) x \leq \bm b_k(t)$, $k\in [p]$ 
that are active
at point $x=\bm u(t)$. Hence
$\bm u(t)$ is an extreme point of $\bm\Pcal(t)$,
and so, $\bm u(t)$ must belong to the set
$\{\bm u^1(t),\dots, \bm u^q(t)\}$. A similar argument
shows that if $\bm v\in \K^d$ generates
an extreme ray of $\bm \Pcal$, the ray generated by $\bm v(t)$
is extreme in $\bm \Pcal(t)$ for $t$ large enough,
and so $\bm v(t)\in \cup_{\ell\in [r]} \R_+ \bm v^\ell(t)$.
It follows that $\bm \Pcal(t) \subset \bm \Qcal(t)$
holds for $t$ large enough.
\end{proof}
\begin{remark}
One can show, by arguments of the same nature as in the latter
proof, that for all $\bm x\in \K^d$,
\begin{align}
\bm x\in \bm \Pcal\iff (\bm x(t)\in \bm \Pcal(t) \text{ for } t\text{ large enough}) \enspace.\label{e-equivpt}
\end{align}
Note, however, that the smallest value $t_0$ such that
$\bm x(t)\in \bm \Pcal(t)$ for all $t\geq t_0$ cannot be
bounded uniformly in $\bm x$.
\end{remark}

\subsection{Tropical polyhedra}\label{subsec:tropicalization}
\noindent
Tropical polyhedra may be informally thought of as the analogues of convex polyhedra over the tropical semifield $\Trop$. 
Note that in this semifield, the zero and unit elements are $-\infty$ and $0$, respectively. Given $\lambda \in \Trop \setminus \{-\infty\}$, we shall also denote by $\lambda^{\tdot (-1)}$ the inverse of $\lambda$ for the tropical multplication, \ie, $\lambda^{\tdot (-1)} \coloneqq -\lambda$.
The tropical addition and multiplication extend to vectors and matrices in the usual way. More precisely, $A \tplus B \coloneqq (A_{ij} \tplus B_{ij})_{ij}$, and $A \tdot B \coloneqq (\tsum_k A_{ik} \tdot B_{kj})_{ij}$, where $A$ and $B$ are two matrices of appropriate sizes with entries in $\Trop$. Further, the $d$-fold Cartesian product $\Trop^d$ is equipped with the structure of semimodule, thanks to the tropical multiplication $\lambda \tdot v \coloneqq (\lambda \tdot v_i)_i$ of a vector $v$ with a scalar~$\lambda$. 

A \emph{tropical halfspace} of $\Trop^d$ is the set of points $x \in \Trop^d$ which satisfy one tropical linear (affine) inequality,
\[
\max( \alpha_1 + x_1, \dots, \alpha_d + x_d, \beta ) \leq \max(\alpha'_1 + x_1, \dots, \alpha'_d + x_d, \beta' ) \,,
\]
where $\alpha, \alpha' \in \Trop^d$ and $\beta, \beta' \in \Trop$. A \emph{tropical polyhedron} is the intersection of finitely many tropical halfspaces. Equivalently, it can be written in the form
\[
\bigl\{ x \in \Trop^d \colon A \tdot x \tplus b \leq A' \tdot x \tplus b' \bigr\}
\] 
where $A, A' \in \Trop^{p \times d}$ and $b, b' \in \Trop^p$ for some $p \geq 0$.  The tropical semifield  $\Trop=\R\cup\{-\infty\}$ is equipped with the order topology, which gives rise to the product topology on $\Trop^d$.  Tropical halfspaces, and thus tropical polyhedra, are closed in this topology.  Note also that the subset topology on $\R^d \subset \Trop^d$ agrees with the usual Euclidean topology.

An analogue of the Minkowski--Weyl Theorem allows for the ``interior represention'' of a tropical polyhedron $\Pcal$ in terms of linear combinations of points and rays~\cite{GaubertKatz2011minimal}.
That is to say that there exist finite sets $U, V \subset \Trop^d$ such that $\Pcal$ is the set of all points of the form 
\begin{align}
\Bigl(\tsum_{u \in U} \alpha_u \tdot u \Bigr) \; \tplus \; \Bigl(\tsum_{v \in V} \beta_v \tdot v \Bigr)
\label{e-minkowskiweyltropical}
\end{align}
where $\alpha_u,\beta_v\in \Trop$ and $\tsum_{u \in U}\alpha_u$ is equal to the tropical unit, \ie, the real number $0$.
We shall say that the tropical polyhedron $\Pcal$ is \emph{generated} by the sets $U$ and $V$. The term $\tsum_{u \in U} \alpha_u \tdot u $ is a tropical convex combination of the points in $U$, while $\tsum_{v \in V} \beta_v \tdot v$ is a tropical linear combination of the vectors in $V$. These are the tropical analogues of convex and conic hulls, respectively. Indeed, all scalars $\alpha_u, \beta_v$ are implicitly non-negative in the tropical sense, \ie, they are greater than or equal to the tropical zero element $-\infty$. We point out that the ``tropical polytopes'' considered by Develin and Sturmfels~\cite{develin2004} are obtained by omitting the term $\tsum_{u \in U} \alpha_u \tdot u $ and by requiring the vectors $v\in V$ to have finite coordinates in the representation~\eqref{e-minkowskiweyltropical}.

If $\Pcal$ is a non-empty tropical polyhedron, the supremum $\sup(u,v) = u \tplus v$ with respect to the partial order $\leq$ of $\Trop^d$ of any two points $u, v \in \Pcal$ 
also belongs to $\Pcal$. If in addition $\Pcal$ is compact, then the supremum of an arbitrary subset of $\Pcal$ is well-defined and
belongs to $\Pcal$. Consequently, there is a unique element in $\Pcal$ which is the coordinate-wise maximum of all elements in
$\Pcal$. We call it the \emph{(tropical) barycenter} of $\Pcal$, as it is the mean of $\Pcal$ with respect to the uniform idempotent measure. 

\begin{figure}
\begin{tikzpicture}[trop/.style={convex,draw=black,very thick},scale=0.6]
\draw[gray!40!] (-0.5,-0.5) grid (9.5,7.5);

\filldraw[orange] (1,5) circle (3pt);
\filldraw[orange] (5,7) circle (3pt);
\filldraw[orange] (1,3) circle (3pt);
\filldraw[orange] (4,1) circle (3pt);
\filldraw[orange] (6,1) circle (3pt);
\filldraw[orange] (8,6) circle (3pt);
\draw[orange,thick] (1,5)  -- (3,5) -- (5,7);
\draw[orange,thick] (1,3) -- (4,3) -- (4,1);
\draw[orange,thick] (6,1) -- (6,4) -- (8,6);
\end{tikzpicture}
\caption{The three possible shapes of tropical segments in dimension $2$}\label{fig:tropical_segments}
\end{figure}
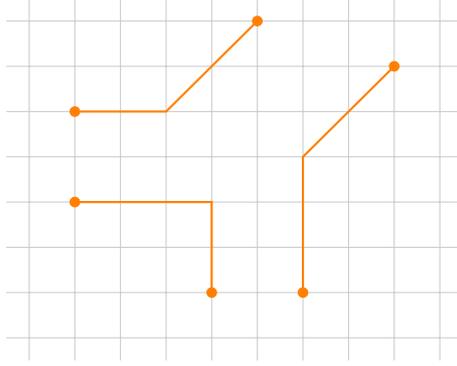

The \emph{tropical segment} between the points $u, v \in \Trop^d$, denoted by $\tsegm(u,v)$, is defined as the set of points of the form $\lambda \tdot u \tplus \mu \tdot v$ such that $\lambda \tplus \mu = 0$. Equivalently, the set $\tsegm(u,v)$ is the tropical polyhedron generated by the sets $U=\{u,v\}$ and $V=\emptyset$. As illustrated in Figure~\ref{fig:tropical_segments}, tropical segments are polygonal curves, and the direction vectors supporting every ordinary segment have their entries in $\{0, \pm 1\}$~\cite[Proposition~3]{develin2004}. We shall slightly refine this statement in the case where $u \leq v$. To this end, for $K\subset[d]$, we denote by $e^K$ the vector whose $k$th entry is equal to $1$ if $k \in K$, and $0$ otherwise.
\begin{lemma}\label{lemma:trop_segments}
Let $u, v \in \Trop^d$ such that $u \leq v$. The tropical segment $\tsegm(u,v)$ is a polygonal curve which, when oriented from $u$ to $v$, consists of ordinary segments supported by direction vectors of the form $e^{K_1}, \dots, e^{K_\ell}$ where $K_1 \subsetneq \dots \subsetneq K_\ell$ and $\ell \leq d$. 
\end{lemma}
\begin{proof}
Since $u \leq v$, the set $\tsegm(u,v)$ is reduced to the set of the points of the form $u \tplus (\mu \tdot v)$, where $\mu \leq 0$. Let $K(\mu)$ be the set of $i \in [d]$ such that $u_i < \mu + v_i$. When $\mu$ ranges from $-\infty$ to $0$, $K(\mu)$ takes a finite number of values $K_0 = \emptyset \subsetneq K_1 \subsetneq \dots \subsetneq K_\ell$, where $\ell \leq d$. It is immediate that the ordinary segments constituting the tropical segment $\tsegm(u,v)$ are supported by the vectors $e^{K_1}, \dots, e^{K_\ell}$. 
\end{proof}

We now relate tropical polyhedra with their classical analogues over Puiseux series via the valuation map. The fact that sums of non-negative Puiseux series do not suffer from cancellation translates into the following.
\begin{lemma}\label{lemma:homomorphism}
The valuation map is a monotone and surjective semifield homomorphism from $\K_+$ to $\Trop$.
\end{lemma}
\begin{proof}
That $\val$ is a homomorphism is a consequence of~\eqref{e-morphism} and the subsequent discussion.
Monotonicity and surjectivity are straightforward.
\end{proof}
This carries over to Puiseux polyhedra in the non-negative orthant:
\begin{proposition}
\label{prop-direct}
The image under the valuation map of any Puiseux polyhedron $\bm \Pcal \subset \K_+^d$ is a tropical polyhedron in $\Trop^d$. 
\end{proposition}
\begin{proof}
Let $\bm U, \bm V \subset \K^d$ be two finite collections of vectors such that $\bm \Pcal$ is the set of combinations of the form
\begin{align}
\bm x = \sum_{\bm u \in \bm U} \bm \alpha_{\bm u} \bm u  + \sum_{\bm v \in \bm V} \bm \beta_{\bm v} \bm v \, ,
\label{e-minkowskiweyl}
\end{align}
where $\bm \alpha_{\bm u}, \bm \beta_{\bm v} \in \K_+$ and $\sum_{\bm u \in \bm U}\bm \alpha_{\bm u} = 1$. Observe that $\bm U$ and $\bm V$ both lie in $\K_+^d$ since $\bm \Pcal \subset \K_+^d$. From Lemma~\ref{lemma:homomorphism}, we deduce that $\val(\bm \Pcal)$ is contained in the tropical polyhedron $\Pcal$ generated by the sets $U \coloneqq \val(\bm U)$ and $V \coloneqq \val(\bm V)$.
Conversely, any point in $\Pcal$ of the form~\eqref{e-minkowskiweyltropical} is the image under the valuation map of 
\[
\sum_{\bm u \in \bm U} \frac{1}{\bm Z} t^{\alpha_u} \bm u  + \sum_{\bm v \in \bm V} t^{\beta_v} \bm v  \, ,
\]
where $\bm Z= \sum_{u \in U} t^{\alpha_u}$ is such that $\val \bm Z=0$.
\end{proof}

The special case of Proposition~\ref{prop-direct} 
concerning ``tropical polytopes'' in the sense of~\cite{develin2004}
was already proved by Develin and Yu \cite[Proposition~2.1]{DevelinYu07}.
One can show that, conversely, each tropical polyhedron arises as the image under the valuation map of a polyhedron included in $\K_+^d$; see \cite[Proposition~2.6]{tropical+simplex}.

\subsection{Metric properties}

In this section, we establish various metric estimates which will be used in the analysis of the central path in Section~\ref{sec-cb}. These estimates involve different metrics. We start with the non-symmetric metric $\funk$, defined by
\[
\funk(x,y) \coloneqq \inf\bigl\{ \rho \geq 0 \colon x+\rho e \geq y \bigr\} \, ,
\]
where $x, y \in \Trop^d$. Recall that $e$ denotes the all-$1$-vector. Writing the inequality $x + \rho e \geq y$ as $\rho \tdot x \geq y$ reveals that $\funk$ is the tropical analogue of the \emph{Funk metric} which appears in Hilbert's geometry~\cite{PT14}. Equivalently, we can write 
\[
\funk(x,y) = \max(0,\max_k (y_k - x_k)) \, ,
\]
with the convention $-\infty + (+\infty) = -\infty$. In this way, we observe that $\funk$ is derived from the one-sided $\ell_\infty$-norm $\|\cdot\|$ which we used to define the wide neighborhood of the central path in~\eqref{eq:wide_neighborhood}, \ie, $\funk(x,y)=\|x-y\|$. We point out that  $\funk(x,y) < + \infty$ if and only if the support of $x$ contains the support of $y$, \ie, $\{ k \colon x_k \neq -\infty \} \supset \{ k \colon y_k \neq -\infty \}$.

The metric $d_\infty$ induced by the ordinary $\ell_\infty$-norm is obtained by symmetrizing $\funk$ as follows:
\[
d_\infty(x,y) \coloneqq \max(\funk(x,y),\funk(y,x)) \, .
\]
We shall consider another symmetrization of $\funk$, leading to the affine version of \emph{Hilbert's projective metric}:
\[
\hilbert(x,y) \coloneqq \funk(x,y) + \funk(y,x) \, . 
\]
The metric $\hilbert$ was shown in~\cite{cgq02} to be the canonical metric in tropical convexity.
For instance, the projection onto a convex set is well defined and is a best approximation in this metric.
The relevance of Hilbert's geometry to the study of the central path was already observed by Bayer and Lagarias~\cite{BayerLagarias89a}. Notice that $\hilbert(x,y) < + \infty$ if and only if the supports of the two vectors $x, y \in \Trop^d$ are identical. 

We extend our notation to sets as follows.
Given $X, Y \subset \Trop^d$ we define
\[ 
\hilbert(X,Y) := \sup_{x \in X} \inf_{y \in Y} \hilbert(x,y)
\quad \text{and} \quad 
d_\infty(X,Y) := \sup_{x \in X} \inf_{y \in Y} d_\infty(x,y)\, .
\]
These are the \emph{directed Hausdorff distances} from $X$ to $Y$ induced by $\hilbert$ and $d_\infty$, respectively.

In order to establish the metric properties of this section, we repeatedly use the following elementary inequalities: if $t > 0$ and $\gamma_1, \dots, \gamma_p \in \R_+$,
\begin{equation}
\max(\log_t \gamma_1, \dots, \log_t \gamma_p) \leq \log_t (\gamma_1 + \dots + \gamma_p) \leq \max(\log_t \gamma_1, \dots, \log_t \gamma_p) + \log_t p \, . \label{eq:maslov}
\end{equation}

We start with a metric estimate over classical and tropical segments.
\begin{lemma}\label{lemma:uniform_convergence_segment}
  Let $S = [u, v]$ be a segment in $\R^d$, and let $S^\trop$ be the tropical segment between the points $\log_t u$ and $\log_t v$.
  Then 
  \[
  d_\infty(S^\trop, \log_t S) \leq \log_t 2 \, ,
  \]
  Here and below $\log_t S$ is short for $\{\log_t s:s\in S\}$.
\end{lemma}
\begin{proof}
Let $x \coloneqq \lambda \tdot (\log_t u) \tplus \mu \tdot (\log_t v)$ be a point of the tropical segment $S^\trop$, where $\lambda, \mu \in \Trop$ are such that $\lambda \tplus \mu = 0$. Now the point 
\[
x' \coloneqq \frac{t^\lambda u + t^\mu v}{t^\lambda + t^\mu}
\]
belongs to $S$. Using~\eqref{eq:maslov} and $\lambda \tplus \mu = 0$, we get $0 \leq \log_t (t^\lambda + t^\mu) \leq \log_t 2$. Similarly, for all $i \in [d]$, we have $x_i \leq \log_t(t^\lambda u_i + t^\mu v_i) \leq x_i + \log_t 2$. We deduce that $x_i - \log_t 2 \leq \log_t x'_i \leq x_i + \log_t 2$. In consequence, $d_\infty(x, \log_t x') \leq \log_t 2$. It follows that $d_\infty(S^\trop, \log_t S) \leq \log_t 2$.
\end{proof}

We now focus on estimating the distance between tropical polyhedra and related logarithmic deformations of convex polyhedra.  To this end, we consider a Puiseux polyhedron $\bm \Pcal$ included in the non-negative orthant, as well as the associated parametric family of polyhedra $\bm \Pcal(t)$ over~$\R$. The following theorem shows that the tropical polyhedron $\val(\bm \Pcal)$ is the log-limit of the polyhedra $\bm \Pcal(t)$, and that the convergence is uniform. This is related to a result of Briec and Horvath, who established in~\cite{BriecHorvath04} a uniform convergence property for a parametric family of convex hulls.
\begin{theorem}
\label{th:uniform_convergence_polyhedra}
Let $\bm \Pcal \subset \K_+^d$ be a Puiseux polyhedron.
Then the sequence $(\log_t \bm \Pcal(t))_t$ of real polyhedra converges to the tropical polyhedron $\val(\bm \Pcal)$ with respect to the directed Hausdorff distance $\hilbert$.   
\end{theorem}

\begin{proof}
By Proposition~\ref{prop-bygenerators}, we can find finite
sets $\bm U ,\bm V \subset \K^d_+$ 
such that for sufficiently large $t$, the real polyhedron $\bm \Pcal(t)$ is generated by the sets of points $\bm U(t) \coloneqq \{ \bm u(t) \colon \bm u \in \bm U\} \subset \R^d_+$ and rays $\bm V(t) \coloneqq \{ \bm v(t) \colon \bm v \in \bm V\} \subset \R^d_+$. 
Let $\bm u \in \bm U$. If $t$ is large enough, then $\bm u_i(t) = 0$ is equivalent to $\val(\bm u)_i = -\infty$, 
for all $i \in [d]$. Thus $\funk(\log_t \bm u(t), \val \bm u)$ as well as $\funk(\val \bm u, \log_t \bm u(t))$ converge to $0$ when $t \to +\infty$. The situation is similar for~$\bm v(t)$ and $\val \bm v$, for any $\bm v \in \bm V$.

Moreover, the tropical polyhedron $\Pcal \coloneqq \val(\bm \Pcal)$ is generated by the sets $\val(\bm U)$ and $\val(\bm V)$, as shown in the proof of Proposition~\ref{prop-direct}.

Now consider $x \in \bm \Pcal(t)$. From Carathéodory's Theorem 
we know that there exist subsets $\{\bm u^k(t)\}_{k \in K} \subset \bm U(t)$ and $\{\bm v^\ell(t)\}_{\ell \in L} \subset \bm V(t)$ with 
$|K| +|L|\leq d+1$ such that the point $x$ can be written as
\[
x = \sum_{k \in K} \alpha_k \bm u^k(t) + \sum_{\ell \in L} \beta_\ell \bm v^\ell(t) \, ,
\]
where $\alpha_k, \beta_\ell > 0$ for $k \in K$, $\ell \in L$ and $\sum_{k \in K} \alpha_k = 1$. Then, for all $i \in [d]$, we get
\begin{multline}
\tsum_{k \in K} \bigl((\log_t \alpha_k) \tdot \log_t \bm u^k(t)\bigr) \tplus \tsum_{\ell \in L} \bigl((\log_t \beta_\ell) \tdot \log_t \bm v^\ell(t)\bigr) \leq \log_t x_i \\
\leq \biggl[\tsum_{k \in K} \bigl((\log_t \alpha_k) \tdot \log_t \bm u^k(t)\bigr) \tplus \tsum_{\ell \in L} \bigl((\log_t \beta_\ell) \tdot \log_t \bm v^\ell(t)\bigr)\biggr] + \log_t (|K| + |L|) \, .
\label{eq:delta_proof_eq1}
\end{multline}
Setting $\gamma \coloneqq \max_{k \in K} \alpha_k$, we have $\frac{1}{|K|} \leq \gamma \leq 1$. Now we define 
\[
x' \coloneqq \bigl(\tsum_{k \in K} \alpha'_k \tdot u^k\bigr) \tplus \bigl(\tsum_{\ell \in L} \beta'_\ell \tdot v^\ell\bigr) \, ,
\]
where $\alpha'_k \coloneqq \log_t (\alpha_k/\gamma)$, $\beta'_\ell \coloneqq \log_t \beta_\ell$, $u^k \coloneqq \val(\bm u^k)$ and $v^\ell \coloneqq \val(\bm v^\ell)$. By choice of $\gamma$ we have $\tsum_{k \in K} \alpha'_k = 0$ and thus $x' \in \Pcal$. Further, $x_i > 0$ if and only if there exists $k \in K$ such that $\bm u^k_i(t) > 0$ or $\ell \in L$ such that $\bm v^\ell_i(t) > 0$. Provided that $t$ is sufficiently large, this is equivalent to the fact  that $u^k_i > -\infty$ for some $k \in K$, or $v^\ell_i > -\infty$ for some $\ell \in L$. This latter property amounts to $x'_i > -\infty$. Consequently, we have $\hilbert(\log_t x, x') < +\infty$, and we can derive from~\eqref{eq:delta_proof_eq1} that
\begin{multline}
x'_i - \max\bigl(\log_t |K|+ \max_{k \in K} \funk(\log_t \bm u^k(t), u^k), \max_{\ell \in L} \funk(\log_t \bm v^\ell(t), v^\ell)\bigr) \leq \log_t x_i \\
 \leq x'_i + \log_t (|K| + |L|) + \max\bigl(\max_{k \in K} \funk(u^k, \log_t \bm u^k(t)), \max_{\ell \in L} \funk(v^\ell, \log_t \bm v^\ell(t))\bigr) \, ,\label{eq:uniform_convergence_polyhedra1}
\end{multline}
for all $i \in [d]$. Finally, we deduce that 
\begin{multline}
 \hilbert(\log_t \bm \Pcal(t), \Pcal) \leq \log_t (d+1) + \max\bigl(\max_{\bm u \in \bm U} \funk(\val \bm u, \log_t \bm u(t)), \max_{\bm v \in \bm V} \funk(\val \bm v, \log_t \bm v(t))\bigr) \\
+ \max\bigl(\log_t (d+1) + \max_{\bm u \in \bm U} \funk(\log_t \bm u(t), \val \bm u), \max_{\bm v \in \bm V} \funk(\log_t \bm v(t), \val \bm v)\bigr) \, , \label{eq:uniform_convergence_polyhedra2}
\end{multline}
which tends to $0$ when $t \to +\infty$. 
\end{proof}
\begin{remark}
For the sake of brevity, we only stated and proved here
the one sided metric estimates which we will use in the proof
of our main results, leaving it to the interested reader to derive
the symmetrical metric estimates. 
For instance, the inequality $d_\infty(\log_t S, S^\trop)
\leq \log_t 2$ can be shown by a method similar to the one of Lemma~\ref{lemma:uniform_convergence_segment}. Similarly, a variant of the proof of Theorem~\ref{th:uniform_convergence_polyhedra} shows that 
the reversed Hausdorff distance
$\hilbert(\val(\bm \Pcal),\log_t \bm \Pcal(t)) $ tends to zero as well
as $t\to\infty$.
\end{remark}

Next we refine the convergence result just obtained by providing a metric estimate in the special case where $\bm \Pcal$ is a polyhedron given by constraints with monomial coefficients. Here a Puiseux series of the form $\pm t^{\alpha}$ is called \emph{monomial}, with the convention $t^{-\infty} = 0$. Further, a vector or a matrix is \emph{monomial} if all its entries are. 
For a matrix $\bm M$ of size $d \times d$ we introduce the quantity $\eta(\bm M) > 0$ by letting
\[
\eta(\bm M) \coloneqq \min \Bigl\{\eta \colon \sigma, \tau \in \Sym{d}, \, \eta = \sum_{i = 1}^d \alpha_{i\sigma(i)} -  \sum_{i = 1}^d \alpha_{i\tau(i)} > 0 \Bigr\} \, ,
\]
where $\Sym{d}$ stands for the symmetric group over $[d]$. We use the convention $\min \emptyset = +\infty$. Phrased differently, the determinant of $\bm M$ is a Puiseux series with finitely many terms with decreasing exponents, and $\eta(\bm M)$ provides a lower bound on the gap between any two successive exponents (if any). This allows us to obtain explicit upper and lower bounds for $\log_t \abs{\det \bm M(t)}$ in terms of $\val(\det \bm M)$. Note that these bounds hold without any assumption on the genericity of the matrix $\bm M$.
\begin{lemma}\label{lemma:det}
Let $\bm M \in \K^{d \times d}$ be a monomial matrix. Then, for all $t > 0$, we have
\begin{align*}
\log_t \abs{\det \bm M(t)} & \leq \val(\det \bm M) + \log_t d! \, ,
\shortintertext{and, if even $t \geq (d!)^{1/\eta(\bm M)}$, then we get}
\val(\det \bm M) & \leq  \log_t \abs{\det \bm M(t)} + \log_t d!  \, .
\end{align*}
\end{lemma}

\begin{proof}
First note that the statement is trivial when $\det \bm M = 0$, since in this case, the determinant of $\bm M(t)$ vanishes for all $t > 0$. 
Now suppose that $\det \bm M \neq 0$.
Since $\bm M$ is monomial, every entry of $\bm M$ is of the form $\epsilon_{ij} t^{\alpha_{ij}}$ where $\epsilon_{ij} \in \{\pm 1\}$.
Therefore, we obtain
\[
\det \bm M = \sum_{k = 1}^p c_k t^{\beta_k} \, ,
\]
where the following conditions are met:
\begin{inparaenum}[(i)]
\item each $c_k$ is a non-null integer, and $\sum_k \abs{c_k} \leq d!$,
\item every $\beta_k$ is of the form $\sum_{i = 1}^d \alpha_{i\sigma(i)}$ for a certain permutation $\sigma \in \Sym{d}$,
\item $\beta_1 > \dots > \beta_p > -\infty$.
\end{inparaenum}
With this notation, we have $\val(\det \bm M) = \beta_1$, and $\beta_i - \beta_{i+1} \geq \eta(\bm M)$. Similarly, for all $t > 0$, we have $\det \bm M(t) = \sum_{k = 1}^p c_k t^{\beta_k}$. This leads to
\[
\log_t \abs{\det \bm M(t)} \leq \beta_1 + \log_t \Bigl(\sum_{k = 1}^p \abs{c_k} t^{\beta_k - \beta_1} \Bigr) \leq \val(\det \bm M) + \log_t d! \, .
\]
Further, provided that $t \geq (d!)^{1/\eta(\bm M)}$, we have
\[
\sum_{k = 2}^p \abs{c_k} t^{\beta_k - \beta_1} \leq (d! - 1) t^{\beta_2 - \beta_1} \leq (d! - 1) t^{-\eta(\bm M)} \leq 1 - 1/d! \, ,
\]
and so
\[
\log_t \abs{\det \bm M(t)} \geq \beta_1 + \log_t \Bigl(1 - \sum_{k = 2}^p \abs{c_k} t^{\beta_k - \beta_1} \Bigr) \geq \val(\det \bm M) - \log_t d! \, . \qedhere
\]
\end{proof}

Recall that we write $e$ for the all-$1$-vector of an appropriate size.
\begin{theorem}\label{th:polyhedron_metric_estimate}
Let $\bm \Pcal \subset \K_+^d$ be a polyhedron of the form $\{ \bm x \in \K^d \colon \bm A \bm x \leq \bm b \}$ where $\bm A$ and $\bm b$ are monomial. Let $\eta_0$ be the minimum of the quantities $\eta(\bm M)$ where $\bm M$ is a square submatrix of $\small\begin{pmatrix} \bm A & \bm b & 0 \\ \transpose{e} & 0 & 1 \end{pmatrix}$ of order $d$. Then, for all $t \geq (d!)^{1/\eta_0}$, we have:
\[
\hilbert(\log_t \bm \Pcal(t), \val (\bm \Pcal)) \leq \log_t \bigl((d+1)^2 (d!)^4\bigr) \, .
\]
\end{theorem}

\begin{proof}
We employ the notation introduced in the proof of Theorem~\ref{th:uniform_convergence_polyhedra}. Note that the inequality given in~\eqref{eq:uniform_convergence_polyhedra2} holds for any sets $\bm U, \bm V$ generating the polyhedron $\bm \Pcal$. In particular, we can set $\bm U$ to the set of vertices of $\bm \Pcal$, and $\bm V$ to a set consisting of precisely one representative of every extreme ray of the recession cone of $\bm \Pcal$.

Let $\bm u \in \bm U$. Since $\bm u$ is a vertex, there exists a subset $I$ of cardinality $d$ such that $\bm A_I \bm u = \bm b_I$, where $\bm A_I$ and $\bm b_I$ consist of the rows of $\bm A$ and $\bm b$, respectively, which are indexed by $i \in I$, and $\bm A_I$ is invertible. Therefore, by Cramer's rule, every coordinate $\bm u_i$ can be expressed as a fraction of the form $\pm \det \bm M / \det \bm A_I$ where $\bm M$ is a submatrix of $\begin{pmatrix} \bm A & \bm b \end{pmatrix}$ of size $d \times d$. Recall that $\bm u_i$ is non-negative, and hence $\bm u_i = \abs{\det \bm M} / \abs{\det \bm A_I}$. By definition, $\eta(\bm M)$ and $\eta(\bm A_I)$ are greater than or equal to $\eta_0$. From Lemma~\ref{lemma:det} we derive that
\[
\val(\bm u_i) - 2 \log_t d! \leq \log_t \bm u_i(t) \leq \val(\bm u_i) + 2 \log_t d! \, ,
\]
for all $t \geq (d!)^{1/\eta_0}$.
Since these inequalities hold for all $i \in [d]$, we deduce that the two quantities $\funk(\val \bm u, \log_t \bm u(t))$ and $\funk(\log_t \bm u(t), \val \bm u)$ differ by at most $2 \log_t d!$. 

The recession cone of $\bm \Pcal$ is the set $\{ \bm z \in \K^d \colon \bm A \bm z \leq 0 \}$. Since $\bm \Pcal$ is contained in the positive orthant, so does the recession cone. Therefore, without loss of generality, we can assume that every $\bm v \in \bm V$ satisfies $\sum_i \bm v_i = 1$. In this way, the elements of $\bm V$ precisely correspond to the vertices of the polyhedron $\{ \bm z \in \K^d \colon \bm A \bm z \leq 0 \, , \transpose{e} \bm z = 1 \}$. Using the same arguments as above, we infer that, for all $\bm v \in \bm V$, the two values $\funk(\val \bm v, \log_t \bm v(t))$ and $\funk(\log_t \bm v(t), \val \bm v)$ differ by at most $2 \log_t d!$, as soon as $t \geq (d!)^{1/\eta_0}$.
Now the claim follows from~\eqref{eq:uniform_convergence_polyhedra2}.
\end{proof}

\section{The Tropical Central Path}
\label{sec:tropical_central_path}
\noindent
Our core idea is to introduce the \emph{tropical central path} of a linear program over Puiseux series.
This is defined as the image of the classical primal-dual central path under the valuation map.
By (\ref{eq:centralpath}) the classical central path is a segment of a real algebraic curve, and so its tropicalization is a segment of a (real) tropical curve and thus piecewise linear.
It will turn out that for certain Puiseux linear programs the tropical central path carries a substantial amount of metric information.
In Sections~\ref{sec:curvature} and~\ref{sec-cb} below we will see that this applies to the linear programs $\realCEX{r}{t}$.
As its key advantage the tropical central path turns out to be much easier to analyze than its classical counterpart.

\subsection{A geometric characterization of the tropical central path}\label{subsec:geometric_characterization}

As in Section~\ref{sec:classical_central_path}, we consider a dual pair of linear programs, except that now the coefficents lie in the field $\K$ of absolutely convergent real Puiseux series from Section~\ref{subsec:fields}:
\begin{alignat*}{6}\label{LP:puiseux}
& \text{minimize} & \quad \scalar{\bm c}{\bm x} & \quad \text{subject to} & \quad \bm A \bm x + \bm w & = \bm b \, , & \;  (\bm x, \bm w) & \in \K_+^N \, , \tag*{$\textbf{LP}(\bm A, \bm b, \bm c)$} \\
& \text{maximize} & \quad \scalar{\bm b}{\bm y} & \quad \text{subject to} & \quad \bm s - \transpose{\bm A} \bm y & = \bm c \, , & \;  (\bm s, \bm y) & \in \K_+^N \, . \tag*{$\textbf{DualLP}(\bm A, \bm b, \bm c)$}
\end{alignat*}
where $\bm A \in \K^{m \times n}$, $\bm b \in \K^m$, $\bm c \in \K^n$ and $N \coloneqq n+m$. Here, the Euclidean scalar product $\langle \cdot, \cdot \rangle$ is extended by setting $\langle \bm u, \bm v \rangle \coloneqq \sum_i \bm u_i \bm v_i$ for any vectors $\bm u, \bm v$ with entries over $\K$. Further, we define
\[
\bm \Pcal \coloneqq \{ (\bm x, \bm w) \in \K^N_+ \colon \bm A \bm x + \bm w = \bm b\} \quad \text{and} \quad \bm \Qcal \coloneqq \{(\bm s, \bm y) \in \K^N_+ \colon  \bm s - \transpose{\bm A} \bm y = \bm c \} \, ,
\]
which correspond to the feasible regions of the two Puiseux linear programs.
Moreover, we let $\bm \Fcal \coloneqq \bm \Pcal \times \bm \Qcal$ be the set of primal-dual feasible points, and $\bm \Fcal^\circ \coloneqq \{\bm z \in \bm \Fcal \colon \bm z > 0 \}$ is the strictly feasible subset. Throughout we will make the following assumption:
\begin{assumption}\label{assump:strictly_feasible}
The set $\bm \Fcal^\circ$ is non-empty.
\end{assumption}
This allows us to define the central path of the Puiseux linear programs $\textbf{LP}(\bm A, \bm b, \bm c)$ and $\textbf{DualLP}(\bm A, \bm b, \bm c)$. Indeed, applying Tarski's principle to the real-closed field $\K$ shows that, for all $\bm \mu \in \K$ such that $\bm \mu > 0$, the system
\begin{equation} \label{eq:puiseux_central_path}
\begin{aligned}
\bm A \bm x + \bm w & = \bm b \\
\bm s-\transpose{\bm A} \bm y & = \bm c \\
\begin{pmatrix}
\bm x \bm s \\
\bm w \bm y    
\end{pmatrix}
& = \bm \mu e \\
\bm x, \bm w, \bm y, \bm s & > 0
\end{aligned} 
\end{equation}
has a unique solution in $\K^{2N}$. We denote this solution as $\bm \Ccal(\bm \mu) = (\bm x^{\bm \mu}, \bm w^{\bm \mu}, \bm s^{\bm \mu}, \bm y^{\bm \mu})$ and refer to it as \emph{point of the central path with parameter} $\bm \mu$. Similarly, by Tarski's principle, $\bm \Fcal \neq \emptyset$ ensures that the two linear programs have the same optimal value $\bm \nu \in \K$. We let $(\bm x^*, \bm w^*) \in \bm \Pcal$ and $(\bm s^*, \bm y^*) \in \bm \Qcal$ be a pair of primal and dual optimal solutions. 
 
Let $\Pcal \subset \Trop^N$, $\Qcal \subset \Trop^N$ and $\Fcal \subset \Trop^{2N}$ the images under the valuation map of the primal and dual feasible polyhedra $\bm \Pcal$, $\bm \Qcal$ and $\bm \Fcal$, respectively. Similarly we write $(x^*, w^*) \in \Trop^N$ and $(s^*, y^*) \in \Trop^N$ for the coordinate-wise valuations of the optimal solutions $(\bm x^*, \bm w^*)$ and $(\bm s^*, \bm y^*)$.

Given a primal-dual feasible point $\bm z = (\bm x, \bm w, \bm s, \bm y) \in \bm \Fcal$, the \emph{duality gap}, denoted by $\gap(\bm z)$, is defined as the difference between the values of the primal and dual objective functions, \ie, $\scalar{\bm c}{\bm x} + \scalar{\bm b}{\bm y}$. We recall that $\gap(\bm z)$ is equivalently given by the \emph{complementarity gap} defined as the sum of the pairwise product of primal/dual variables:
\[
\gap(\bm z) = \scalar{\bm x}{\bm s} + \scalar{\bm w}{\bm y} \, .
\]
Observe that the right-hand side of the latter equality consists of sums of non-negatives terms (of the form $\bm x_j \bm s_j$ and $\bm w_i \bm y_i$). Consequently, we can apply the valuation map term-wise and define, for all $z = (x, w, s, y) \in \Fcal$, the \emph{tropical duality gap} as 
\[
\tgap(z) \coloneqq \tscalar{x}{s} \tplus \tscalar{w}{y} \, ,
\]
where $\tscalar{\cdot}{\cdot}$ stands for the tropical analogue of the scalar product, \ie,  $\tscalar{u}{v} \coloneqq \tsum_i (u_i \tdot v_i)$. Then the quantity $\tgap(z)$ equals the valuation of the duality gap of any primal-dual feasible $\bm z$ with $\val(\bm z) = z$.
The study of the tropical central path will require the following \emph{tropical sublevel set} induced by the tropical duality gap
\[
\Fcal^\lambda \coloneqq \{z \in \Fcal \colon \tgap(z) \leq \lambda \} \, ,
\]
which is defined for any $\lambda\in\R$.
We collect a few basic facts about this collection of sublevel sets.
\begin{proposition}\label{prop:sublevel}
Let $\bm \mu \in \K$ such that $\bm \mu > 0$, and $\lambda = \val(\bm \mu)$. Then
\begin{enumerate}[(i)]
\item\label{item:sublevel1} 
the set $\Fcal^\lambda$ is a bounded tropical polyhedron given by $\Pcal^\lambda \times \Qcal^\lambda$ where
\begin{align*}
\Pcal^\lambda & \coloneqq \{(x,w) \in \Pcal \colon \tscalar{s^*}{x} \tplus \tscalar{y^*}{w}  
\leq \lambda \} \, , \\
\Qcal^\lambda & \coloneqq \{(s,y) \in \Qcal \colon \tscalar{x^*}{s} \tplus \tscalar{w^*}{y} 
\leq \lambda \} \, ,
\end{align*}
\item\label{item:sublevel2} and the image under $\val$ of the point $\bm \Ccal(\bm \mu)$ lies in $\Fcal^\lambda$.
\end{enumerate} 
\end{proposition}
 
\begin{proof}
We start with the proof of~\eqref{item:sublevel2}. By definition, $\bm \Ccal(\bm \mu) \in \bm \Fcal$ so that $\val (\bm \Ccal(\bm \mu)) \in \Fcal$. Moreover, we have $\tgap\bigl(\val (\bm \Ccal(\bm \mu)) \bigr) = \val (\gap(\bm \Ccal(\bm \mu))$. Since $\gap(\bm \Ccal(\bm \mu)) = N \bm \mu$ by \eqref{eq:duality_measure}, we deduce that the previous quantity is equal to the valuation of $N \bm \mu$, which is $\lambda$. This forces that $\val(\bm \Ccal(\bm \mu))$ lies in $\Fcal^\lambda$. 

We need to show that $\Fcal^\lambda$ is a tropical polyhedron which is bounded.
To this end we consider $\bm z = (\bm x, \bm w, \bm s, \bm y) \in \bm \Fcal$.
Recall that $\bm \nu$ is the common optimal value of the primal and the dual Puiseux linear programs. Then we obtain
\[
\gap(\bm z) = \bigl(\scalar{\bm c}{\bm x} - \bm \nu\bigr) + \bigl(\scalar{\bm b}{\bm y} + \bm \nu\bigr) = \gap(\bm x, \bm w, \bm s^*, \bm y^*) + \gap(\bm x^*, \bm w^*, \bm s, \bm y) \, .
\]
Since the right-hand side of this identity is a sum of two non-negative terms, applying the valuation map yields $\tgap(z) = \tgap(x, w, s^*, y^*) \tplus \tgap(x^*, w^*, s, y)$ for all $z = (x, w, s, y) \in \Fcal$.
Thus we can express the tropical sublevel set $\Fcal^\lambda$ as
\[
\Fcal^\lambda = \{ (x, w, s, y) \in \Fcal \colon \tgap(x, w, s^*, y^*) \tplus \tgap(x^*, w^*, s, y) \leq \lambda \} = P^\lambda \times Q^\lambda \, .
\]
In particular, the set $\Fcal^\lambda$ is a tropical polyhedron.

To complete our proof we still have to show that $\Fcal^\lambda$ is bounded. To this end, pick some point $z^\circ = (x^\circ, w^\circ, s^\circ, y^\circ)$ in $\Fcal^\lambda$ which has finite coordinates. For instance, we can take $z^\circ = \val(\bm \Ccal(\bm \mu))$ by making use of~\eqref{item:sublevel2}. Now consider an arbitrary point $(x, w, s, y) \in \Fcal^\lambda$. As $\Fcal^\lambda=\Pcal^\lambda\times\Qcal^\lambda$, we know that $(x, w, s^\circ, y^\circ) \in \Fcal^\lambda$. In particular, $\tgap(x, w, s^\circ, y^\circ) \leq \lambda$, or, equivalently, 
\[
\left\{
\begin{aligned}
x_j \tdot s^\circ_j & \leq \lambda \qquad \text{for all} \ j \in [n] \, ,\\
w_i \tdot y^\circ_i & \leq \lambda \qquad \text{for all} \ i \in [m] \, .
\end{aligned}\right.
\]
Since $s^\circ_j, y^\circ_i > -\infty$, this entails that $x_j \leq \lambda \tdot (s^\circ_j)^{\tdot (-1)}$ and $w_i \leq \lambda \tdot (y^\circ_i)^{\tdot (-1)}$ for all $i \in [m]$ and $j \in [n]$. Similarly, the entries of $(s,y)$ are bounded, too. This completes the proof of~\eqref{item:sublevel1}.
\end{proof}

As a bounded tropical polyhedron the set $\Fcal^\lambda$ admits a (tropical) barycenter.
Recall that the latter was defined as the coordinate-wise maximum of that set.
The following theorem relates this barycenter with the valuation of the central path, and gives rise to the definition of the tropical central path:
\begin{theorem}\label{th:trop_central_path}
The image under the valuation map of the central path 
of the pair of primal-dual linear programs $\textbf{LP}(\bm A, \bm b, \bm c)$ and $\textbf{DualLP}(\bm A, \bm b, \bm c)$ can be described by:
\begin{equation}\label{eq:trop_central_path:barycenter}
  \val ( \bm \Ccal (\bm \mu)) = \text{barycenter of } \{ z \in \val (\bm{\mathcal{F}} ) \colon \tgap(z) \leq \val( \bm \mu) \} \ ,
\end{equation}
for any $\bm \mu \in \K$ such that $\bm \mu > 0$.
\end{theorem}

\begin{proof}
Let $\lambda \coloneqq \val(\bm \mu)$, and denote by $\bar z = (\bar x, \bar w,\bar s,\bar y)$ the barycenter of the tropical polyhedron $\Fcal^\lambda$. By construction we have $\val(\bm \Ccal (\bm \mu)) \leq \bar z$. Moreover, since $\tgap\bigl(\bar z\bigr) \leq \lambda$, we also have $\bar x_j \tdot \bar s_j \leq \lambda$ and $\bar w_i \tdot \bar y_i \leq \lambda$ for all $i \in [m]$ and $j \in [n]$. It follows that:
\begin{equation}\label{eq:trop_central_path:barycenter:characterization}
\left\{
\begin{aligned}
\lambda & = \val(\bm x^{\bm \mu}_j \bm s^{\bm \mu}_j) = \val(\bm x^{\bm \mu}_j) \tdot \val(\bm s^{\bm \mu}_j) \leq \bar x_j \tdot \bar s_j \leq \lambda \, , \\
\lambda & = \val(\bm w^{\bm \mu}_j \bm y^{\bm \mu}_j) = \val(\bm w^{\bm \mu}_i) \tdot \val(\bm y^{\bm \mu}_i) \leq \bar w_i \tdot \bar y_i \leq \lambda \, .
\end{aligned}\right.
\end{equation}
As a consequence, the inequality $\val(\bm \Ccal (\bm \mu)) \leq \bar z$ is necessarily an equality. 
\end{proof}

The quantity \eqref{eq:trop_central_path:barycenter}, which depends only on the valuation of $\bm \mu$, is called \emph{the tropical central path} at $\lambda = \val( \bm \mu)$ and is denoted by 
\[ \Ctrop(\lambda) = (x^\lambda, w^\lambda, s^\lambda, y^\lambda)
\enspace .
\]
Analogously, the \emph{primal} and the \emph{dual tropical central paths} are defined by projecting to $(x^\lambda, w^\lambda)$ and $(s^\lambda, y^\lambda)$, respectively. As shown in \eqref{eq:trop_central_path:barycenter:characterization}, the primal and dual components of the tropical central path are characterized by
\begin{equation}\label{eq:trop_csc}
x^\lambda_j \tdot s^\lambda_j  = \lambda = w^\lambda_i \tdot y^\lambda_i
\end{equation}
for all $i\in[m]$ and $j\in[n]$.
The next statement shows that the tropical central path is a polygonal curve with a particularly simple structure.
\begin{proposition}\label{prop:piecewise}
The tropical central path $\lambda \mapsto \Ctrop(\lambda)$ is a monotone piecewise linear function,
whose derivative at each smooth point is a vector of the 
form $(e^K, e^{[N]\setminus K})$, for some $K \subset [N]$.
\end{proposition}
\begin{proof}
Let us denote by $g \colon \Trop^N \to \Trop$ the function which sends $(x,w)$ to $\tscalar{s^*}{x} \tplus \tscalar{y^*}{w}$. Pick finite generating sets $U \subset \Trop^N$ and $V \subset \Trop^N \setminus \{(-\infty, \dots, -\infty)\}$ for the tropical polyhedron $\Pcal$. Since $(x^\lambda, w^\lambda)$ lies in $\Pcal$ it can be expressed as
\[
(x^\lambda, w^\lambda) = \Bigl(\tsum_{u \in U} \alpha_u \tdot u\Bigr) \tplus \Bigl(\tsum_{v \in V} \beta_v \tdot v \Bigr) \, ,
\]
where $\tsum_{u \in U} \alpha_u = 0$. The inequality $g(x^\lambda, w^\lambda) \leq \lambda$ now amounts to $\alpha_u \tdot g(u) \leq \lambda$ and $\beta_v \tdot g(v) \leq \lambda$ for all $u \in U$ and $v \in V$. As $(x^\lambda, w^\lambda)$ is the barycenter of $\Pcal^\lambda$, the coefficients $\alpha_u$ and $\beta_v$ 
can be chosen 
to be maximal. This enforces $\alpha_u = \min\bigl(0, \lambda \tdot (g(u))^{\tdot (-1)} \bigr)$ and $\beta_v = \lambda \tdot (g(v))^{\tdot (-1)}$, using the convention $(-\infty)^{\tdot (-1)} = +\infty$. Note that $g(v) \neq -\infty$ for all $v \in V$. Indeed, if there were a ray $v \in V$ with $g(v) = -\infty$, then any point of the form $(x^*, w^*) \tplus (\beta \tdot v)$ would belong to $\Pcal^\lambda$.  The latter conclusion would contradict the boundedness of $\Pcal^\lambda$. Therefore, all $\alpha_u$ and $\beta_v$ belong to $\R$. Observe also that $\alpha_u$ and $\beta_v$, thought of as functions of $\lambda$, are monotone, piecewise linear, and that their derivatives at any smooth point take value in $\{0,1\}$. 
It follows that $\lambda \mapsto (x^\lambda, s^\lambda)$ is piecewise linear and monotone, and that its derivative at any smooth point is of the form $e^K$. 

A similar argument reveals that the map $\lambda \mapsto (s^\lambda, y^\lambda)$ is also piecewise linear and monotone, with a derivative at any smooth point of the form $e^{K'}$ 
for some set $K' \subset [N]$. 
In consequence, the derivative of the map $\Ctrop$ at any smooth point is of the form $(e^K, e^{K'})$ and, from \eqref{eq:trop_csc}, we get that $K' = [N]\setminus K$.  This proves our claim.
\end{proof}

We end this section with the following direct consequence of Proposition~\ref{prop:piecewise}.
\begin{corollary}\label{cor:nonexpansive}
If $\lambda \leq \lambda'$, then $\Ctrop(\lambda) \leq \Ctrop(\lambda') \leq \Ctrop(\lambda) + (\lambda' - \lambda) e$.
\end{corollary}

\subsection{A uniform metric estimate on the convergence of the central path}\label{subsec:uniform}

For now, we have related the tropical central path with the central path of the linear programs $\textbf{LP}(\bm A, \bm b, \bm c)$ and $\textbf{DualLP}(\bm A, \bm b, \bm c)$ over Puiseux series. These give rise to a parametric family of dual linear programs $\text{LP}(\bm A(t), \bm b(t), \bm c(t))$ and $\text{DualLP}(\bm A(t), \bm b(t), \bm c(t))$ over the reals. The purpose of this section is to relate the resulting family of central paths with the tropical central path. More precisely, we will show that the logarithmic deformation of these central paths uniformly converge to the tropical curve. In fact, we will even show that the logarithmic deformation of the wide neighborhoods $\Ncal^{-\infty}_{\theta}$ of these central paths collapses onto the tropical central path.

Let $\bm \Pcal(t), \bm \Qcal(t) \subset \R^N$ be the feasible sets of $\text{LP}(\bm A(t), \bm b(t), \bm c(t))$ and $\text{DualLP}(\bm A(t), \bm b(t), \bm c(t))$, respectively. (Note that this notation is compatible with the one introduced in 
Section~\ref{subsec:fields} thanks to Proposition~\ref{prop-bygenerators}.) Then $\bm \Fcal(t) \coloneqq \bm \Pcal(t) \times \bm \Qcal(t)$ is the primal-dual feasible set, while $\bm \Fcal^\circ(t) \coloneqq \{z \in \bm \Fcal(t) \colon z > 0 \}$ comprises only those primal-dual points which are strictly feasible.
The following lemma relates the optimal solutions of $\textbf{LP}(\bm A, \bm b, \bm c)$ and $\textbf{DualLP}(\bm A, \bm b, \bm c)$ with 
the ones
of $\text{LP}(\bm A(t), \bm b(t), \bm c(t))$ and $\text{DualLP}(\bm A(t), \bm b(t), \bm c(t))$.
\begin{lemma}\label{lemma:archimedean}
There exists a positive real number $t_0$ such that for all $t > t_0$ the following three properties hold:
\begin{enumerate}[(i)]
\item\label{item:int} the set $\bm \Fcal^\circ(t)$ is non-empty;
\item the number $\bm \nu(t)$ is the optimal value of $\text{LP}(\bm A(t), \bm b(t), \bm c(t))$ and $\text{DualLP}(\bm A(t), \bm b(t), \bm c(t))$;
\item $(\bm x^*(t), \bm w^*(t))$ and $(\bm s^*(t), \bm y^*(t))$ constitute optimal solutions.
\end{enumerate}
\end{lemma}
\begin{proof}
  For two series $\bm u, \bm v \in \K$ the equality $\bm u = \bm v$ forces $\bm u(t) = \bm v(t)$ for all $t$ large enough.
  A similar statement holds for inequalities like $\bm u \leq \bm v$ and $\bm u < \bm v$.
  We infer that the set $\bm \Fcal^\circ(t)$ is not empty if $t \gg 1$.
  Moreover, $(\bm x^*(t), \bm w^*(t)) \in \Pcal(t)$, $(\bm s^*(t), \bm y^*(t)) \in \Qcal(t)$ and $\scalar{\bm c(t)}{\bm x^*(t)} = \scalar{\bm b(t)}{\bm y^*(t)} = \bm \nu(t)$.
  Since $\text{LP}(\bm A(t), \bm b(t), \bm c(t))$ and $\text{DualLP}(\bm A(t), \bm b(t), \bm c(t))$ are dual to one another, we conclude that $(\bm x^*(t), \bm w^*(t))$ and $(\bm s^*(t), \bm y^*(t))$ form a pair of optimal solutions.
\end{proof}

Throughout the following we will keep that value $t_0$ from Lemma~\ref{lemma:archimedean}.  When $t > t_0$, we know from~Lemma~\ref{lemma:archimedean}\eqref{item:int} that the primal-dual central path of the linear programs $\text{LP}(\bm A(t), \bm b(t), \bm c(t))$ and $\text{DualLP}(\bm A(t), \bm b(t), \bm c(t))$ is well defined. In this case we denote by $\Ccal_t(\mu)$ the point of this central path with parameter $\mu$, where $\mu \in \R$ and $\mu > 0$.  Let us fix the real precision parameter $\theta$ in the open interval from $0$ to~$1$.  Then the set
\[
\Ncal^{-\infty}_{\theta,t}(\mu) \coloneqq 
\Bigl\{ z = (x,w,s,y) \in {\bm \Fcal}^\circ(t) \colon \bar{\mu}(z) = \mu \; \text{and}\;  
\begin{pmatrix}
 x s \\
 w y 
\end{pmatrix} \geq (1 - \theta) \mu e \Bigr\}
\]
is a neighborhood of the point $\Ccal_t(\mu)$.
A direct inspection shows that the union of the sets $\Ncal^{-\infty}_{\theta,t}(\mu)$ for $\mu > 0$ agrees with the wide neighborhood $\Ncal^{-\infty}_\theta$ of the entire central path of the linear program $\text{LP}(\bm A(t), \bm b(t), \bm c(t))$ over $\R$; see \eqref{eq:wide_neighborhood}.
In order to stress the dependence on $t$, we denote this neighborhood by $\Ncal^{-\infty}_{\theta,t}$. With this notation, we have
\[
\Ncal^{-\infty}_{\theta,t} = \bigcup_{\mu > 0} \Ncal^{-\infty}_{\theta,t}(\mu) \, .
\]
Further let
\begin{equation}\label{eq:delta_t}
  \delta(t) \coloneqq 2\hilbert(\log_t \bm \Fcal(t), \Fcal) \, ,
\end{equation}
which, by Theorem~\ref{th:uniform_convergence_polyhedra}, tends to $0$ when $t$ goes to $+\infty$.
The following result states that we can uniformly bound the distance from the image of $\Ncal^{-\infty}_{\theta,t}(\mu)$ under $\log_t$ to the point $ \Ctrop(\log_t \mu)$ of the tropical central path, independently of $\mu$:
\begin{theorem}\label{th:uniform_convergence_central_path}
For all $t > t_0$ and $\mu > 0$ we have
\[
d_\infty\bigl(\log_t \Ncal^{-\infty}_{\theta,t}(\mu), \Ctrop(\log_t \mu) \bigr) \leq \log_t \Bigl(\frac{N}{1-\theta}\Bigr) + \delta(t) \, . 
\]
\end{theorem}

\begin{proof}
Choose $t > t_0$ and $\mu > 0$, and $z = (x,w,s,y) \in \Ncal^{-\infty}_{\theta,t}(\mu)$.  Letting $\lambda \coloneqq \log_t \mu$ we claim that it suffices to prove that
\begin{equation}\label{eq:upper_bound}
\log_t z \leq \Ctrop(\lambda) + (\log_t N + \delta(t)) e  \, .
\end{equation}
Indeed, by definition of the wide neighborhood $\Ncal^{-\infty}_{\theta,t}(\mu)$, we have $\log_t (x, w)  \geq - \log_t (s, y) + \bigl(\lambda + \log_t (1-\theta)\bigr)e$. Using~\eqref{eq:upper_bound}, we obtain
\[
\log_t (x, w) \geq - (s^\lambda, y^\lambda) + \biggl(\lambda - \log_t \Bigl(\frac{N}{1-\theta}\Bigr) - \delta(t)\biggr)e =  (x^\lambda, w^\lambda) - \biggl(\log_t \Bigl(\frac{N}{1-\theta}\Bigr) + \delta(t)\biggr)e \, ,
\]
where the last equality is due to~\eqref{eq:trop_csc}. Analogously, we can prove that $\log_t (s, y) \geq (s^\lambda, y^\lambda) - \Bigl(\log_t \bigl(\frac{N}{1-\theta}\bigr) + \delta(t)\Bigr) e$.

Now let us show that~\eqref{eq:upper_bound} holds. By definition of the duality measure $\mgap(z)$, we have $\gap(z) = N \mgap(z) = N \mu$. Applying the map $\log_t$ yields
\begin{equation}
\tgap(\log_t z) \leq \log_t \gap(z) =  \lambda + \log_t N \, , \label{eq:uniform_proof1}
\end{equation}
where the inequality is a consequence of the first inequality in~\eqref{eq:maslov}.
 
Let $z' \in \Fcal$ such that $\hilbert(\log_t z, z') < +\infty$. 
Recall that
\begin{equation}\label{eq:funk_sandwich}
  z' - \funk(\log_t z, z') e \leq \log_t z \leq z' + \funk(z', \log_t z) e \, .
\end{equation}
The first inequality in~\eqref{eq:funk_sandwich} gives $\tgap(z') \leq \tgap(\log_t z) + 2\funk(\log_t z, z')$. 
In combination with~\eqref{eq:uniform_proof1} this shows that $z'$ lies in $\Fcal^{\lambda'}$ for $\lambda' \coloneqq \lambda + \log_t N + 2\delta_F(\log_t z, z')$. 
The second inequality in~\eqref{eq:funk_sandwich} now yields
\[
\log_t z \leq z' + \funk(z', \log_t z) e \leq \Ctrop(\lambda') + \funk(z', \log_t z) e  \leq \Ctrop(\lambda) + \bigl(\log_t N + 2 \hilbert(\log_t z, z')\bigr)e  \, ,
\]
where the second inequality follows from $\Ctrop(\lambda')$ being the barycenter of $\Fcal^{\lambda'}$, and the last inequality is a consequence of Corollary~\ref{cor:nonexpansive}. As this argument is valid for all $z' \in \Fcal$ within a finite distance from $\log_t z$ we obtain that $\log_t z \leq \Ctrop(\lambda) + \bigl(\log_t N + \delta(t)\bigr) e$.
\end{proof}

\subsection{Main example}\label{subsec:cex_tropical_central_path}

The family $\realCEX{r}{t}$ of linear programs over the reals from the introduction may also be read as a linear program over the field $\K$, thinking of $t$ as a formal parameter. We denote this linear program by $\puiseuxCEX{r}$. The goal of this section is to obtain a complete description of the corresponding tropical central path.

Introducing slack variables $\bm w_1, \dots, \bm w_{3r-1}$ in the first $3r-1$ inequalities of $\puiseuxCEX{r}$, and adding the redundant inequalities $\bm x_i\geq 0$ for $1\leq i<2r-1$, gives rise to a linear program $\puiseuxCEXslack{r}$, which is of the form $\LP(\bm A, \bm b, \bm c)$ in dimension $N = 5r-1$. (Note that the last two inequalities of $\puiseuxCEX{r}$ are non-negativity constraints, which is why we do not need slack variables for them.)
The dual Puiseux linear program (with slacks) is referred to as $\dualPuiseuxCEXslack{r}$. We retain the notation introduced in Section~\ref{subsec:geometric_characterization}; for instance, we denote by $\bm \Pcal$ and $\bm \Qcal$ the primal and dual feasible sets respectively.

To begin with, we verify that Assumption~\ref{assump:strictly_feasible} is satisfied. Due to the lower triangular nature of the system of inequalities in $\puiseuxCEX{r}$, we can easily find a vector $\bm x$ satisfying every inequality of this system in a strict manner. In other words, we can find $(\bm x, \bm w) \in \bm \Pcal$ such that $\bm x > 0$ and $\bm w > 0$. Moreover, since the feasible set of $\puiseuxCEX{r}$ is bounded, the set $\bm \Pcal$ is bounded as well. This implies that the dual feasible set $\bm \Qcal$ contains a point $(\bm s, \bm y)$ satisfying $\bm s > 0$, $\bm y > 0$. As a result, the set $\bm \Fcal^\circ$ is non-empty. 

We focus on the description of the primal part $\lambda \mapsto (x^\lambda, w^\lambda)$ of the tropical central path, since the dual part can be readily obtained by using the relations~\eqref{eq:trop_csc}. It can be checked that the optimal value of $\puiseuxCEXslack{r}$, and subsequently of $\dualPuiseuxCEXslack{r}$, is equal to $0$. Since in our case, the primal objective vector $\bm c$ is given by the nonnegative vector $(1, 0, \dots, 0) \in \K^n$, we deduce that we can choose the dual optimal solution $(\bm s^*, \bm y^*)$ as $(1, 0, \dots, 0) \in \K^{n+m}$. As a consequence of Theorem~\ref{th:trop_central_path} and Proposition~\ref{prop:sublevel}\eqref{item:sublevel1}, the point $(x^\lambda, w^\lambda)$ on the primal tropical central path agrees with the barycenter of the tropical sublevel set
\begin{equation}\label{eq:primal_sublevel_set}
  \Pcal^\lambda = \{(x,w) \in \Pcal \colon x_1 \leq \lambda \} \, .
\end{equation}
Recall that $\Pcal$ stands for $\val(\bm \Pcal)$.

We first restrict our attention to the $x$-component of the tropical central path.  To this end, let $\bm \Pcalx$ be the projection of the primal feasible set $\bm \Pcal$ onto the coordinates $\bm x_1, \dots, \bm x_{2r}$. This is precisely the feasible set of the Puiseux linear program~$\puiseuxCEX{r}$. Further, let $\Pcalx$ be the image under $\val$ of $\bm\Pcalx$. Equivalently, this is the projection of $\Pcal$ onto $x_1, \dots, x_{2r}$.
We claim that $\Pcalx$ is given by the $3r+1$ tropical linear inequalities
\begin{equation}\label{eq:Pcalx}
\begin{lgathered}
x_1 \leq 2 \, , \; x_2 \leq 1 \\
x_{2j+1} \leq 1 + x_{2j-1} \, , \;  x_{2j+1} \leq 1 + x_{2j} \tikzmark{} \\
x_{2j+2} \leq (1-1/2^j) + \max(x_{2j-1}, x_{2j}) \tikzmark{}
\end{lgathered}
\insertbigbrace{$1 \leq j < r$ \, ,} 
\end{equation}
which are obtained by applying the valuation map to the inequalities in $\puiseuxCEX{r}$ coefficient-wise. While this can be checked by hand, we can also apply~\cite[Corollary~14]{tropical_spectrahedra}, as $\Pcalx$ is a regular set in $\Trop^{2r}$, \ie, it coincides with the closure of its interior.

By~\eqref{eq:primal_sublevel_set} we deduce that the point $x^\lambda$ is the barycenter of the tropical polyhedron $\{ x \in \Pcalx \colon x_1 \leq \lambda \}$. We arrive at the following explicit description of $x^\lambda$.
\begin{proposition}\label{prop:counterexample_tropical_central_path}
For all $\lambda \in \R$, the point $x^\lambda$ is given by the recursion
\begin{align*}
x^\lambda_1 & = \min(\lambda,2) \\
x^\lambda_2 & = 1 \\
x^\lambda_{2j+1} & = 1 + \min(x^\lambda_{2j-1}, x^\lambda_{2j}) \tikzmark{} \\
x^\lambda_{2j+2} & = (1-1/2^j) + \max(x^\lambda_{2j-1}, x^\lambda_{2j}) \tikzmark{} 
\insertbigbrace{$1 \leq j < r$ \, .}
\end{align*}
\end{proposition}
\begin{proof}
We introduce the family of maps $F_j \colon (a,b) \mapsto (1 + \min(a,b), 1 - 1/2^j + \max(a,b))$ where $1 \leq j < r$.
With this notation, the point $x$ lies in $\Pcalx$ and also satisfies $x_1 \leq \lambda$ if and only if
\begin{equation}\label{eq:Pcalx_dyn}
x_1 \leq \min(\lambda, 2) \, , \quad x_2 \leq 1 \, , \quad (x^\lambda_{2j+1}, x^\lambda_{2j+2}) \leq F_j(x^\lambda_{2j-1}, x^\lambda_{2j}) \, , 
\end{equation}
for every $1 \leq j < r$. Since the maps $F_j$ are order preserving, the barycenter of the tropical polyhedron defined by~\eqref{eq:Pcalx_dyn} is the point which attains equality in~\eqref{eq:Pcalx_dyn}.
\end{proof}
Observe that the map $\lambda \mapsto x^\lambda$ is constant on the interval $[2, \infty[$, while it is linear on $]{-\infty}, 0]$. In constrast, on the remaining interval $[0,2]$, the shape of this map is much more complicated. This is illustrated in Figure~\ref{fig:x_central_path}. 

\begin{figure}[t]
\begin{center}
\begin{tikzpicture}
\begin{scope}
[>=stealth',scale=1.75,curve/.style={black, thick}]
\draw[help lines, gray!40!] (-0.05,-0.05) grid (2.25,5.25); 
\foreach \x in {0,...,2} { \node [anchor=north] at (\x,0) {\x}; }
\foreach \y in {0,...,5} { \node [anchor=east] at (-0.05,\y) {\y}; }

\draw[gray!40!, ->] (-0.05,0) -- (2.25,0) node[black,above right={-5pt and -1pt}] {$\lambda$};

\draw[curve, dashdotted] (-0.1,-0.1) -- (1.5,1.5) node[below right={-5pt and -1pt}] {$x^\lambda_1$} -- (2,2) -- (2.1,2); 
\draw[curve, dashdotted, red] (-.1,1) -- (2.1,1) node[below right={-5pt and -1pt}] {$x^\lambda_2$}; 

\foreach \n in {1,...,4} {
	\pgfmathsetmacro{\range}{pow(2,\n-1)-1}
		\foreach \k in {0,...,\range} {
    		\draw[curve] ( 4 * \k / 2^\n   , \n + 2 * \k / 2^\n + 1 / 2^\n) 
 			-- ( 4 * \k  / 2^\n  + 2 / 2^\n , \n + 2 * \k / 2^\n + 1 / 2^\n)
			-- ( 4 * \k  / 2^\n  + 4 / 2^\n, \n + 2 * \k / 2^\n  + 3 / 2^\n);
 			\draw[curve, red] ( 4 * \k / 2^\n   , \n + 2 * \k / 2^\n) 			--  ( 4 * \k  / 2^\n  + 2 / 2^\n , \n + 2 * \k / 2^\n + 2 / 2^\n )
			--  ( 4 * \k  / 2^\n  + 4 / 2^\n, \n + 2 * \k / 2^\n  + 2 / 2^\n);
   		} 
	\pgfmathsetmacro{\nodd}{int(2*\n+1)}
	\pgfmathsetmacro{\neven}{int(2*(\n+1))}
	\draw[curve,red] (2, \n+1) -- (2.1, \n+1) node[below right={-5pt and -1pt}] {$x^\lambda_{\nodd}$};
	\draw[curve] (2, \n + 1 + 1/2^\n) -- (2.1, \n + 1 + 1/2^\n) node[above right={-5pt and -1pt}] {$x^\lambda_{\neven}$};

	\draw[curve,red] (0, \n) -- (-.1, \n - .1);
	\draw[curve] (0, \n  + 1/2^\n) -- (-.1, \n  + 1/2^\n);
}
\end{scope}
\begin{scope}[shift={(7,1)},>=stealth',scale=1,  curve/.style={blue!50, very thick},tube/.style={color=lightgray, fill opacity = 0.7}]
     \draw[help lines, gray!40!] (0,0) grid (7,7); 
     \draw[gray!40!, ->] (0,0) -- (7.5,0);
     \draw[gray!40!, ->] (0,0) -- (0,7.5);
     \node[anchor = north] at (7.5,0) {$x_{2r-1}$};
     \node[anchor = west] at (0,7.5) {$x_{2r}$};
\node [anchor=north,font=\small] at (0,0) {$r-1$};
\node [anchor=east,font=\small] at (0,0) {$r-1$};
 \foreach \x in {2,4,...,6} { \node [anchor=north,font=\small] at ($(\x,0) +(0.1,0)$) {$\begin{multlined}(r-1)\\+\tfrac{\x}{2^{r-1}}\end{multlined}$}; 
}
 \foreach \y in {1,3,...,8} { \node [anchor=east,font=\small] at ($(0,\y)+(0,-0.1)$) {$\begin{multlined}(r-1)\\+\tfrac{\y}{2^{r-1}}\end{multlined}$}; }

 \draw[dashdotted] (0,0) -- (7,7);

\filldraw[curve] (0,1) circle (1.5pt) node[font=\small,above right] {$\lambda = 0$} -- (2,1) circle (1.5pt) node[font=\small,below right] {$\lambda = \frac{1}{2^{r-1}}$} -- (2,3) circle (1.5pt) node[font=\small, above left] {$\lambda = \frac{2}{2^{r-1}}$} -- (4,3) circle (1.5pt) node[font=\small, below right] {$\lambda = \frac{3}{2^{r-1}}$} -- (4,5) circle (1.5pt) node[font=\small, above left] {$\lambda = \frac{4}{2^{r-1}}$} -- (6,5) circle (1.5pt) node[font=\small, below right] {$\lambda = \frac{5}{2^{r-1}}$} -- (6,7) circle (1.5pt) node[font=\small, right] {$\lambda = \frac{6}{2^{r-1}}$};
\end{scope}
\end{tikzpicture}
\end{center}
\caption{Left: the $x$-components of the primal tropical central path of $\puiseuxCEX{r}$ for $r \geq 5$ and $0 \leq \lambda \leq 2$. Right: the projection of the tropical central path of $\puiseuxCEX{r}$ onto the  $(x_{2r-1}, x_{2r})$-plane.}\label{fig:x_central_path}
\end{figure}
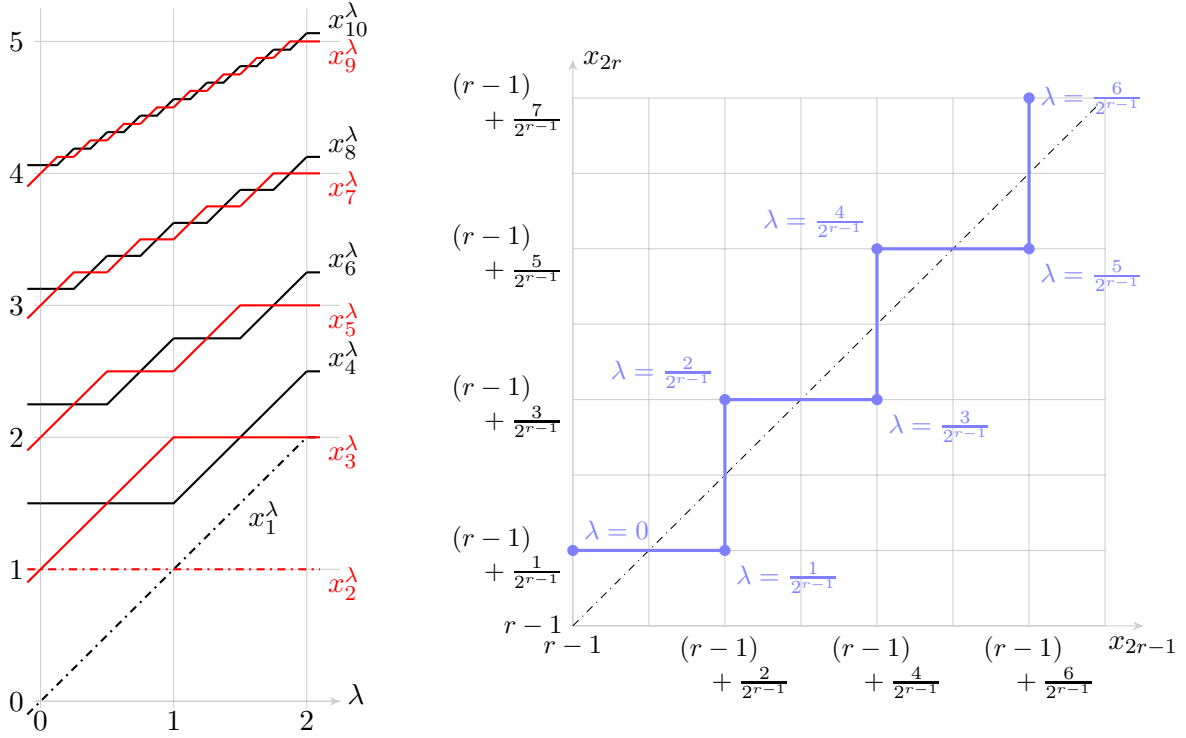

We now incorporate the slack variables $\bm w_1, \dots, \bm w_{3r-1}$ into our analysis.  The points $(\bm x,\bm w)$ in the primal feasible set $\bm \Pcal$ is defined by the following constraints:
\begin{equation}
\begin{aligned}
\bm x_1 + \bm w_1 & = t^2 \\
\bm x_2 + \bm w_2 & = t \\
\bm x_{2j+1} + \bm w_{3j} & = t \, \bm x_{2j-1} \tikzmark{} \\
\bm x_{2j+1} + \bm w_{3j+1} & = t\, \bm x_{2j}  \\
\bm x_{2j+2} + \bm w_{3j+2} & = t^{1-1/2^j} (\bm x_{2j-1} + \bm x_{2j}) \tikzmark{} \\
(\bm x, \bm w) & \in \K^N_+
\end{aligned} \label{eq:primal_feasible_set_with_slack}
\insertbigbrace{$1 \leq j < r$ \, .}
\end{equation}
This entails that the points $(x,w)$ in $\Pcal = \val(\bm \Pcal)$ satisfy the inequalities
\begin{equation}
\begin{aligned}
& w_1 \leq 2 \, , \; w_2 \leq 1 \\
& w_{3j} \leq 1 + x_{2j-1} \, , \;  w_{3j+1} \leq 1 + x_{2j} \tikzmark{} \\
& w_{3j+2} \leq (1-1/2^j) + \max(x_{2j-1}, x_{2j}) \tikzmark{} 
\end{aligned} \label{eq:ineq_slack}
\insertbigbrace{$1 \leq j < r$ \, .}
\end{equation}

The following result states that all these inequalities are tight for all points $(x^\lambda, w^\lambda)$ on the primal tropical central path.
\begin{proposition}\label{prop:central_path_slack}
For all $\lambda \in \R$, the point $w^\lambda$ is described by the following relations:
\begin{equation}
\begin{aligned}
w^\lambda_1 & = 2 \, , \; w^\lambda_2 = 1 \\
w^\lambda_{3j} & = 1 + x^\lambda_{2j-1} \tikzmark{} \\
w^\lambda_{3j+1} & = 1 + x^\lambda_{2j} \\
w^\lambda_{3j+2} & = (1-1/2^j) + \max(x^\lambda_{2j-1}, x^\lambda_{2j}) = x^\lambda_{2j+2} \tikzmark{} 
\end{aligned} \label{eq:central_path_slack}
\insertbigbrace{$1 \leq j < r$}
\end{equation} 
\end{proposition}

\begin{proof}
Let $\bar w$ be the element defined by the relations in~\eqref{eq:central_path_slack}. We want to prove that $w^\lambda = \bar w$. By Theorem~\ref{th:trop_central_path} and Proposition~\ref{prop:sublevel}\eqref{item:sublevel1}, it suffices to show that the point $(x^\lambda, \bar w)$ is the barycenter of the tropical polyhedron $\Pcal^\lambda$. Given $(x, w) \in \Pcal$, we have $x \leq x^\lambda$ as $x$ belongs to $\Pcal^\lambda$ and $x^\lambda$ is the barycenter of the latter set. Moreover, $w$ satisfies the inequalities given in~\eqref{eq:ineq_slack}. We deduce that $w \leq \bar w$. 

It now remains to show that $(x^\lambda, \bar w)$ belongs to $\Pcal$, since this immediately leads to $(x^\lambda, \bar w) \in \Pcal^\lambda$. In other terms, we want to find a point $(\bm x, \bm w) \in \bm \Pcal$ such that $\val(\bm x, \bm w) = (x^\lambda, \bar w)$. Let us fix a sequence of positive numbers $\alpha_0 =\frac{1}{2}>\alpha_1>\dots>\alpha_{r-1} >0$.  We claim that letting
\begin{equation}\label{eq:admissible_lift}
\bm x_{2j+1} \coloneqq \alpha_j t^{x_{2j+1}^\lambda} \, , \qquad
\bm x_{2j+2} \coloneqq \alpha_j t^{x_{2j+2}^\lambda} \, , \qquad (0 \leq j < r) 
\end{equation}
and defining $\bm w$ in terms of the equalities in~\eqref{eq:primal_feasible_set_with_slack}, yields such an admissible lift. 

First, observe that $\bm x \geq 0$ and $\val(\bm x) = x^\lambda$.  Second, we have $\bm w_1 = t^2 - \alpha_1 t^{x_1^\lambda}$ and $\bm w_2 = \alpha_0 t$.  Recall that $x_1^\lambda \leq 2$.  Thus, $\bm w_1$ and $\bm w_2$ are non-negative, and they satisfy $\val \bm w_1 = 2$ and $\val \bm w_2 = 1$.  Now, let us consider $j$ for $1\leq j < r$.  We have
\[ 
\bm w_{3j} = t \, \bm x_{2j-1} - \bm x_{2j+1} = \alpha_{j-1} t^{1+x_{2j-1}^\lambda} - \alpha_j t^{x_{2j+1}^\lambda} 
= (\alpha_{j-1}-\alpha_j) t^{\bar w_{3j}} + \alpha_j (t^{\bar w_{3j}}-t^{x_{2j+1}^\lambda}) \,.
\]
As $\bar w_{3j} = 1 + x^\lambda_{2j-1} \geq x_{2j+1}^\lambda$, we have $0\leq t^{\bar w_{3j}}-t^{x_{2j+1}^\lambda} \leq t^{\bar w_{3j}}$, and this gives us $\bm w_{3j} \geq 0$ and $\val \bm w_{3j} = \bar w_{3j}$.  A similar argument shows that $\bm w_{3j+1} \geq 0$ and $\val \bm w_{3j+1} = \bar w_{3j+1}$. 
Finally, we can write
\begin{align*}
\bm w_{3j+2} & = t^{1 - 1/2^j} (\bm x_{2j-1} + \bm x_{2j}) -\bm x_{2j+2} \\
& = 
\begin{cases}
(2 \alpha_{j-1} - \alpha_j) t^{x_{2j+2}^\lambda} & \text{if} \; x_{2j-1}^\lambda = x_{2j}^\lambda \, , \\
(\alpha_{j-1} - \alpha_j) t^{x_{2j+2}^\lambda} + o(t^{x_{2j+2}^\lambda}) & \text{otherwise.}
\end{cases}
\end{align*}
Since $2 \alpha_{j-1} > \alpha_{j-1} > \alpha_j$, we obtain that $\bm w_{3j+2} \geq 0$, and $\val \bm w_{3j+2} = x_{2j+2}^\lambda = \bar w_{3j+2}$.
\end{proof}

Table~\ref{tab:table} gives a summary of the values of the coordinates of the primal tropical central path for specific values of $\lambda$ which we shall use below. 

\begin{table}[tb]
\caption{Coordinates of points on the primal tropical central path of $\puiseuxCEX{r}$ for some specific values of $\lambda$,
  where $1 \leq j < r$ and $k = 0, 2, \dots, 2^{j-1}-2$.}
\label{tab:table}
\vskip-0.5cm
\[
\renewcommand{\arraystretch}{1.75}
\begin{array}{c@{\quad}c@{\quad}c@{\quad}c@{\quad}c@{\quad}c}
\toprule
\lambda & \frac{4k}{2^j} & \frac{4k+2}{2^j} & \frac{4k+4}{2^j} & \frac{4k+6}{2^j} & \frac{4k+8}{2^j} \\
\midrule
x_{2j+1} & j + \frac{2k}{2^j} & j + \frac{2k+2}{2^j} & j + \frac{2k+2}{2^j} & j + \frac{2k+4}{2^j} & j + \frac{2k+4}{2^j} \\
x_{2j+2} & j + \frac{2k+1}{2^j} & j + \frac{2k+1}{2^j} & j + \frac{2k+3}{2^j} & j + \frac{2k+3}{2^j} & j + \frac{2k+5}{2^j} \\
w_{3j} &  j + \frac{2k}{2^j} & j + \frac{2k+2}{2^j} & j + \frac{2k+4}{2^j} & j + \frac{2k+4}{2^j} & j + \frac{2k+4}{2^j} \\
w_{3j+1} & j + \frac{2k+2}{2^j} & j + \frac{2k+2}{2^j} & j + \frac{2k+2}{2^j} & j + \frac{2k+4}{2^j} & j + \frac{2k+6}{2^j} \\
w_{3j+2} & j + \frac{2k+1}{2^j} & j + \frac{2k+1}{2^j} & j + \frac{2k+3}{2^j} & j + \frac{2k+3}{2^j} & j + \frac{2k+5}{2^j} \\
\bottomrule
\end{array}
\]
\end{table}

\section{Curvature Analysis}\label{sec:curvature}
\noindent
The purpose of this section is to show how the combinatorial analysis of the tropical central path translates into lower bounds on the total curvature of the central path of a parametric family of linear programs over the reals.  Our main application will be a detailed version of Theorem~\ref{thm:curvature:intro} from the introduction, and a proof of this result.

Let us recall some basic facts concerning total curvature. For two non-null vectors $x, y \in \R^d$ we denote by $\angle xy$ the measure $\alpha \in [0,\pi]$ of the angle of the vectors $x$ and $y$, so that
\[
\cos \alpha = \frac{\scalar{x}{y}}{\norm{x} \norm{y}} \, ,
\]
where $\norm{\cdot}$ refers to the Euclidean norm.  Given three points $U, V, W \in \R^d$ such that $U \neq V$ and $V \neq W$, we extend this notation to write $\angle U V W$ for the angle formed by the vectors $U V$ and $V W$. 
If $\tau$ is a polygonal curve in $\R^d$ parameterized over an interval $[a,b]$, the \emph{total curvature} $\kappa(\tau,[a,b])$ is defined as the sum of angles between the consecutive segments of the curve. More generally, the total curvature $\kappa(\sigma,[a,b])$ can be defined for an arbitrary curve $\sigma$, parameterized over the same interval, as the supremum of $\kappa(\tau,[a,b])$ over all polygonal curves $\tau$ inscribed in $\sigma$. If $\sigma$ is twice continuously differentiable, this coincides with the standard definition of the total curvature $\int_a^b \norm{\kappa''(s)}ds$, when $\kappa$ is parameterized by arc length; see~\cite[Chapter~V]{nonsmoothpaths} for more background.

Our approach is based on estimating the curvature of the central path using approximations by polygonal curves.  Our first observation is concerned with limits of angles between families of vectors arising from vectors over $\K$. 
\begin{lemma}\label{lemma:angle}
Let $\bm x, \bm y$ be two non-null vectors in $\K^d$, and let $x \coloneqq \val(\bm x)$ and $y \coloneqq \val(\bm y)$. Then the limit of $\angle\bm x(t) \bm y(t)$ for $t\to+\infty$ exists.  Moreover, if the sets $\argmax_{i \in [d]} x_i$ and $\argmax_{i \in [d]} y_i$ are disjoint, then we have
\[
\lim_{t \to +\infty} \angle\bm x(t) \bm y(t) = \frac{\pi}{2} \, .
\]
\end{lemma}
\begin{proof}
Since the field $\mathbb{K}$ is real closed, the Euclidean norm $\norm{\cdot}$ can be extended to a function from $\K^d$ to $\K$ by $\norm{\bm u} \coloneqq \sqrt{\sum_i \bm u_i^2}$ for all $\bm u \in \K^d$. As a consequence, the quotient $\scalar{\bm x}{\bm y}/\bigl(\norm{\bm x} \norm{\bm y}\bigr)$ is an element of $\K$. Let $\alpha$ be its valuation. We obtain
\[
\alpha \leq \max_{i \in [d]} (x_i + y_i) - \bigl(\max_{i \in [d]} x_i + \max_{i \in [d]} y_i\bigr) \leq 0 \, .
\]
Suppose without loss of generality that
$\scalar{\bm x}{\bm y}\neq 0$.
Then, there exists a non-zero number $c \in \R$ such that
\[
\scalar{\bm x(t)}{\bm y(t)}/\bigl(\norm{\bm x(t)} \norm{\bm y(t)}\bigr) = c t^\alpha + o(t^\alpha)
\]
when $t \to +\infty$.
If $\argmax_{i \in [d]} x_i \cap \argmax_{i \in [d]} y_i = \emptyset$, then $\alpha < 0$, implying that the limit of 
$\cos\angle x(t)y(t)=\scalar{\bm x(t)}{\bm y(t)}/\bigl(\norm{\bm x(t)} \norm{\bm y(t)}\bigr)$ as $t\to+\infty$ is equal to $0$.
\end{proof}

We will use Lemma~\ref{lemma:angle} in order to estimate the limit when $t \to +\infty$ of the angle between segments formed by triplets of successive points of the tropical central path. One remarkable property is that the tropical central path of any Puiseux linear program is monotone; see Proposition~\ref{prop:piecewise}.  We refine Lemma~\ref{lemma:angle} to fit this setting.
\begin{lemma}\label{lemma:angle2}
Let $\bm U, \bm V, \bm W \in \K^d$, and $U \coloneqq \val(\bm U)$, $V \coloneqq \val(\bm V)$ and $W \coloneqq \val(\bm W)$. If $\max_{i \in [d]} U_i < \max_{i \in [d]} V_i < \max_{i \in [d]} W_i$, and the sets $\argmax_{i \in [d]} V_i$ and $\argmax_{i \in [d]} W_i$ are disjoint, we have
\[
\lim_{t \to +\infty} \angle\bm U(t) \bm V(t) \bm W(t) = \frac{\pi}{2} \, .
\]
\end{lemma}
\begin{proof}
Let us remark that for all $i \in [d]$, we have $\val(\bm V_i - \bm U_i) \leq \max(U_i, V_i)$, and this inequality is an equality if $U_i \neq V_i$. Since $\max_{i \in [d]} U_i < \max_{i \in [d]} V_i$, we deduce that $\max_{i \in [d]} \val(\bm V_i - \bm U_i) = \max_{i \in [d]} V_i$, and that the argument of the two maxima are equal. The same applies to the coordinates of the vector $\val(\bm W - \bm V)$. We infer from Lemma~\ref{lemma:angle} that $\angle\bm U(t) \bm V(t) \bm W(t)$ tends to $\pi / 2$ whenever $\argmax_{i \in [d]} V_i \cap \argmax_{i \in [d]} W_i = \emptyset$.
\end{proof}

This motivates us to introduce a \emph{(weak) tropical angle}
\[
\angle^* U V W \coloneqq \begin{cases} \frac{\pi}{2} & \text{if $U, V, W$ satisfy the conditions of Lemma~\ref{lemma:angle2}} \, , \\ 0 & \text{otherwise} \end{cases}
\]
for any three points $U, V, W \in \Trop^d$.
We now consider a Puiseux linear program of the form $\textbf{LP}(\bm A, \bm b, \bm c)$, together with the associated family of linear programs $\text{LP}(\bm A(t), \bm b(t), \bm c(t))$ and their primal-dual central path $\Ccal_t$. We obtain that:
\begin{proposition}\label{prop:curvature}
Let $\final, \initial \in \R$ and $\lambda_0 = \final < \lambda_1 < \dots < \lambda_{p-1} < \lambda_p = \initial$. Then
\[
\liminf_{t\to\infty}\kappa\Bigl(\cpath_t,\bigl[t^{\final}, t^{\initial}\bigr]\Bigr) \geq \sum_{k = 1}^{p-1} \angle^* \troppath(\lambda_{k-1}) \troppath(\lambda_k) \troppath(\lambda_{k+1}) \enspace.
\] 
\end{proposition}
\begin{proof}
Let $\lambda \in \R$, and let $\bm U$ be the point of the central path of $\textbf{LP}(\bm A, \bm b, \bm c)$ with parameter $\bm \mu$ equal to the Puiseux series $t^\lambda$. If $t$ is substituted by a sufficiently large real number, the points $\Ccal_t(t^\lambda)$ and $\bm U(t)$ are identical, since both satisfy the constraints given in~\eqref{eq:classical_central_path} for $A = \bm A(t)$, $b = \bm b(t)$, $c = \bm c(t)$ and $\mu = t^\lambda$. 

The monotonicity of the tropical central path shown in Proposition~\ref{prop:piecewise} allows to apply Lemma~\ref{lemma:angle2}. From the previous discussion, we get that
\[
\lim_{t\to\infty} \angle \Ccal_t(t^{\lambda_{k-1}}) \Ccal_t(t^{\lambda_k}) \Ccal_t(t^{\lambda_{k+1}}) \geq \angle^* \troppath(\lambda_{k-1}) \troppath(\lambda_k) \troppath(\lambda_{k+1}) \, .
\]
for all $k \in [p-1]$.
Since the total curvature can be approximated from below by measuring angles of polygonal paths we obtain $\kappa\bigl(\cpath_t,[t^{\final}, t^{\initial}]\bigr) \geq \sum_{k = 1}^{p-1} \angle \Ccal_t(t^{\lambda_{k-1}}) \Ccal_t(t^{\lambda_k}) \Ccal_t(t^{\lambda_{k+1}})$.
\end{proof}

We denote by $\realCEXslack{r}{t}$ and $\dualRealCEXslack{r}{t}$ the linear programs over $\R$ obtained by substituting the parameter $t$ with a real value in the Puiseux linear programs $\puiseuxCEXslack{r}$ and $\dualPuiseuxCEXslack{r}$, respectively.
We are now ready to state and prove the following 
detailed version of Theorem~\ref{thm:curvature:intro}.

\begin{theorem}\label{th:curvature}
For all $\epsilon>0$, 
  the total curvature of the primal central path of the linear program $\realCEXslack{r}{t}$ is greater than 
$(2^{r-2}-1)\frac{\pi}{2}-\epsilon$, provided that $t>0$ is  sufficiently large.  
Moreover, the same holds for the primal-dual central path.
\end{theorem}

\begin{proof}
We will use Proposition~\ref{prop:curvature} to provide a lower bound on $\liminf_{t\to\infty} \kappa(\Ccal_t, [0,2])$ by considering the subdivision of the closed interval $[0,2]$ by the scalars $\lambda_k = \frac{4 k}{2^{r-1}}$ for $k = 0, 1, \dots, 2^{r-2}$.

Let us first point out that, given $\lambda \in [0,2]$, all the dual components of the point $\troppath(\lambda)$ of the tropical central path are less than or equal to $\max(0,\lambda - 1)$. This is a consequence of the identity~\eqref{eq:trop_csc} and the fact that all the primal components are greater than or equal to $\min(1,\lambda)$; see Proposition~\ref{prop:counterexample_tropical_central_path} and~\ref{prop:central_path_slack}. It follows that the dual components are dominated by the primal ones, and thus it suffices to estimate the total curvature of the primal central path.

Using Table~\ref{tab:table}, we deduce that the maximal component of the vector $\troppath(\lambda_k)$ is equal to $r-1 + \frac{2k+2}{2^{r-1}}$, and that is uniquely attained by the coordinate $w_{3(r-1)}(\lambda)$ when $k$ is odd, and by $w_{3(r-1)+1}(\lambda)$ when $k$ is even. This implies $\angle^* \troppath(\lambda_{k-1}) \troppath(\lambda_k) \troppath(\lambda_{k+1}) = \frac{\pi}{2}$, and we obtain the claim from Proposition~\ref{prop:curvature}.
\end{proof}

\begin{remark}
One can refine Theorem~\ref{th:curvature} to additionally obtain a lower bound on the curvature of the dual central path at the same time.  This requires to consider a slightly modified version of  $\realCEXslack{r}{t}$.  More precisely, it can be shown that it suffices to add the constraints $x_{2r+1} + w_{3r} = \frac{1}{t^r} x_{2r-1}$ and $x_{2r+2} + w_{3r+1} = \frac{1}{t^r} x_{2r}$ involving the two extra variables $x_{2r+1}$, $x_{2r+2}$ and the slack variables $w_{3r}$ and $w_{3r+1}$.
\end{remark}

\begin{remark}
Let us compare the lower bound of Theorem~\ref{th:curvature} with the upper bound of Dedieu, Malajovich and Shub~\cite{dedieu2005curvature} obtained from averaging.
Given a real $m{\times} n$ matrix $A$, vectors $b\in \R^m$ and $c\in \R^n$, and an $m{\times} m$ diagonal matrix $E$ with diagonal entries $\pm 1$, we consider the linear program
\begin{align*}
P_E \qquad \min c^\top x, \; Ax -s =b, \; E s\geq 0 \, .
\end{align*}
It is shown there that the sum of the total curvatures of the dual central paths of the $2^m$ linear programs $P_E$ arising from the various choices of sign matrices $E$ does not exceed
\begin{equation}
2\pi n {{m-1}\choose{n}} \,.
\label{e-shub2}
\end{equation}
It can be verified that the dual linear program $\dualRealCEXslack{r}{t}$ is of the form $P_E$ for $E=-I$, where $I$ is the identity matrix, $n=3r+1$ and $m=5r-1$. 
By applying Stirling's formula to~\eqref{e-shub2} we see that the sum of the total curvatures of the dual central paths of the $2^m$ linear programs $P_E$, arising from varying $E$, is bounded by
\[
2\pi (3r+1) {{5r-2}\choose{3r+1}} = O\bigg(\sqrt{r} \Big(\frac{3125}{108}\Big)^r\bigg) \, .
\]
The lower bound of order $\Omega(2^{r})$ from Theorem~\ref{th:curvature} shows that the total curvature of the dual central path of at least one of these $2^m$ linear programs is exponential in $r$.
\end{remark}

\section{Tropical Lower Bound on the Complexity of Interior Point Methods}
\label{sec-cb}
\noindent
In this section, we derive a general lower bound on the number of iterations of interior point methods with a log-barrier. That lower bound is given by the smallest number of tropical segments needed to describe the tropical central path, see Theorem~\ref{th:tropical_lower_bound}. Applying this result to the parametric family of linear programs $\puiseuxCEXslack{r}(t)$ provides a proof of Theorem~\ref{thm:complexity:intro}.

\subsection{Approximating the tropical central path by tropical segments}\label{subsec:piecewise_approximation}
We return to the general situation from Section~\ref{subsec:geometric_characterization}, and consider  a dual pair of linear programs $\textbf{LP}(\bm A, \bm b, \bm c)$ and $\textbf{DualLP}(\bm A, \bm b, \bm c)$ over Puiseux series. 
Lemma~\ref{lemma:trop_segments} and Proposition~\ref{prop:piecewise} yield that the tropical central path can be described as a concatenation of finitely many tropical segments. Given $\final, \initial \in \R$ such that $\final \leq \initial$, we let $\gamma\bigl([\final, \initial]\bigr)$ the smallest number of tropical segments needed to describe the section $\Ctrop\bigl([\final, \initial]\bigr)$ of the tropical central path.

Let $\epsilon > 0$. For $z \in \R^{2N}$ we denote by $\Bcal_\infty(z; \epsilon)$ the closed $d_\infty$-ball centered at $z$ and with radius $\epsilon$.
Further, we fix $\final, \initial \in \R$ such that $\final \leq \initial$.  The set
\[
\Tcal([\final, \initial]; \epsilon) \coloneqq \bigcup_{\final \leq \lambda \leq \initial} \Bcal_\infty(\Ctrop(\lambda); \epsilon)
\]
is the \emph{tubular neighborhood} of the section $\Ctrop([\final, \initial])$ of the tropical central path; see Figure~\ref{fig:tube}. 
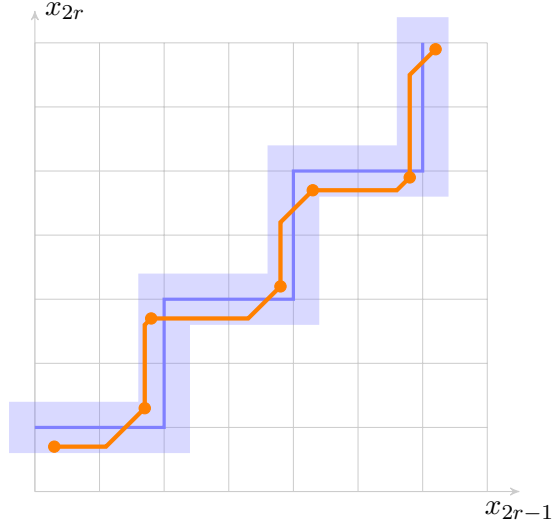
\begin{figure}
\begin{center}
\begin{tikzpicture}[>=stealth',scale=.85,  curve/.style={blue!50, very thick},tube/.style={color=lightgray, fill opacity = 0.7}]
     \draw[help lines, gray!40!] (0,0) grid (7,7); 
     \draw[gray!40!, ->] (0,0) -- (7.5,0);
     \draw[gray!40!, ->] (0,0) -- (0,7.5);
     \node[anchor = north] at (7.5,0) {$x_{2r-1}$};
     \node[anchor = west] at (0,7.5) {$x_{2r}$};

\draw[curve] (0,1) -- (2,1)  -- (2,3) -- (4,3) -- (4,5) -- (6,5) -- (6,7);
\fill[curve,fill opacity=0.3] (-0.4, 1.4) -- (0,1.4) -- (1.6,1.4) -- (1.6,3.4) -- (3.6,3.4) -- (3.6,5.4) -- (5.6,5.4) -- (5.6,7.4) -- (6.4,7.4) -- (6.4,4.6) -- (4.4,4.6) -- (4.4,2.6) -- (2.4,2.6) -- (2.4,0.6) -- (-0.4,0.6) -- cycle;

\draw[orange,ultra thick] (6.2,6.9) -- (5.8,6.5) -- (5.8,4.9) -- (5.6,4.7) -- (4.3,4.7) -- (3.8,4.2) -- (3.8,3.2) -- (3.3,2.7) -- (1.8,2.7) -- (1.7,2.6) -- (1.7,1.3) -- (1.1,0.7) -- (0.3,0.7);

\coordinate (p1) at (6.2,6.9);
\coordinate (p2) at (5.8,4.9);
\coordinate (p3) at (4.3,4.7);
\coordinate (p4) at (3.8,3.2);
\coordinate (p5) at (1.8,2.7);
\coordinate (p6) at (1.7,1.3);
\coordinate (p7) at (0.3,0.7);

\filldraw[orange] (p1) circle (2.5pt);
\filldraw[orange] (p2) circle (2.5pt);
\filldraw[orange] (p3) circle (2.5pt);
\filldraw[orange] (p4) circle (2.5pt);
\filldraw[orange] (p5) circle (2.5pt);
\filldraw[orange] (p6) circle (2.5pt);
\filldraw[orange] (p7) circle (2.5pt);
\end{tikzpicture}
\end{center}
\caption{A tubular neighborhood of the tropical central path (in light blue), containing an approximation by tropical segments (in orange).}\label{fig:tube} 
\end{figure}

Let us consider the union $\Scal$ of a finite sequence of consecutive tropical segments
\[
\Scal \coloneqq \tsegm(z^0, z^1) \cup \tsegm(z^1, z^2) \cup \dots \cup \tsegm(z^{p-1}, z^p) \qquad (z^0, \dots, z^p \in \Trop^{2N})
\]
which is contained in the tubular neighborhood $\Tcal \coloneqq \Tcal([\final, \initial]; \epsilon)$, and which further satisfies $z^0 \in \Bcal_\infty(\Ctrop(\final); \epsilon)$ and $z^p \in \Bcal_\infty(\Ctrop(\initial); \epsilon)$.  That is, $\Scal$ approximates the tropical central path by $p$ tropical segments, starting and ending in small neighborhoods of $z^0$ and $z^p$, respectively; see Figure~\ref{fig:tube} for an illustration. Next we will show that, in this situation, the number of tropical segments in $\Scal$ is bounded from below by $\gamma\bigl([\final, \initial]\bigr)$, provided that the tubular neighborhood $\Tcal$ is tight enough. To this end, we set $\epsilon_0 > 0$ to one sixth of the minimal $d_\infty$-distance between any two distinct vertices in the polygonal curve $\Ctrop(\R)$.  Note that $\epsilon_0$ does not depend on the choices of $\final$ and $\initial$.
\begin{proposition}\label{prop:piecewise_approximation}
  If $\epsilon < \epsilon_0$ then $p \geq \gamma\bigl([\final, \initial]\bigr)$.
\end{proposition}

\begin{proof}
We abbreviate $\gamma \coloneqq \gamma\bigl([\final, \initial]\bigr)$.
Let us consider a sequence $\lambda_0 = \final < \lambda_1 < \dots < \lambda_\gamma = \initial$ such that $\Ctrop([\final, \initial])$ can be decomposed as the union of $\gamma$ successive tropical segments, \ie, 
\[
\Ctrop([\final, \initial]) = \tsegm(\Ctrop(\lambda_0), \Ctrop(\lambda_1)) \cup \dots \cup \tsegm(\Ctrop(\lambda_{\gamma-1}, \lambda_\gamma)) \, .
\]
By definition of $\gamma$, these segments are maximal in the sense that none of them is properly contained in a tropical segment contained in $\Ctrop([\final, \initial])$. 

Let us look at the shape of the tropical central path in the neighborhood of an intermediate point $\Ctrop(\lambda_i)$ for $0 < i < \gamma$. Without loss of generality, we assume that $\Ctrop(\lambda_i) = 0$. Since the segment $\tsegm(\Ctrop(\lambda_{i-1}), \Ctrop(\lambda_i))$ is maximal, the parameter $\lambda_i$ marks a point where the tropical central path is not differentiable. Let $K, L \subset [2N]$ such that $e^K$ and $e^L$ are the left and right derivatives of $\Ctrop$ at $\lambda_i$, respectively. Since $\epsilon<\epsilon_0$, the point $\Ctrop(\lambda_i)$ is the only breakpoint of $\Ctrop([\final, \initial])$ contained in the ball $\Bcal_\infty(\Ctrop(\lambda_i); 3\epsilon)$.  We derive
\begin{equation}\label{eq:breakpoint_neighborhood}
\Ctrop(\lambda + \lambda_i) = 
\begin{cases}	
- \lambda e^K & \text{if} \; -3\epsilon \leq \lambda \leq 0 \, , \\ 
\lambda e^L & \text{if} \; 0 \leq \lambda \leq 3 \epsilon \, . 	
\end{cases}
\end{equation}
Moreover, we have $K \not \subset L$, since $K \subset L$ would contradict the maximality of the tropical segment $\tsegm(\Ctrop(\lambda_{i-1}), \Ctrop(\lambda_i))$, see Lemma~\ref{lemma:trop_segments}. Further, by Proposition~\ref{prop:piecewise}, the set $L$ is non-empty. Thus, let us consider $k \in K \setminus L$ and $\ell \in L$. We introduce the tropical halfspace
\begin{equation}\label{eq:transversal_halfspace}
\Hcal \coloneqq \{ z \in \Trop^{2N} \colon \max(0, z_\ell - \epsilon) \leq z_k + \epsilon \} \, 
\end{equation}
and refer to Figure~\ref{fig:breakpoint_neighborhood} for an illustration of the setting.

We claim that at least one point $z^j$ belongs to the neighborhood $\Tcal \cap \Hcal$ of $\Ctrop(\lambda_i)$.  Arguing indirectly, we assume that $z^j\not\in\Hcal$ for all $j$. The complement of $\Hcal$ is a tropically convex set, \ie, if $z, z' \not \in \Hcal$, then the tropical segment between $z$ and $z'$ is contained in the complement of $\Hcal$. It follows that $\Scal \subset \Tcal \setminus \Hcal$. As $\Bcal_\infty(\Ctrop(\lambda_i); \epsilon) \subset \Hcal$ the points $z^0$ and $z^p$ are located in the same path-connected component of $\Tcal \setminus \Bcal_\infty(\Ctrop(\lambda_i); \epsilon)$.  Further, there is a path from $z^0$ to $\Ctrop(\final)$ in $\Bcal_\infty(\Ctrop(\final); \epsilon) \subset \Tcal$.  That path lies in $\Tcal \setminus \Bcal_\infty(\Ctrop(\lambda_i); \epsilon)$ as $\Bcal_\infty(\Ctrop(\final); \epsilon) \cap \Bcal_\infty(\Ctrop(\lambda_i); \epsilon) = \emptyset$ due to our assumption $\epsilon < \epsilon_0$.  The same argument applies to $z^p$ and $\Ctrop(\initial)$.  Therefore, $\Ctrop(\final)$ and $\Ctrop(\initial)$ belong to the same path-connected component of $\Tcal \setminus \Bcal_\infty(\Ctrop(\lambda_i); \epsilon)$.  However, the latter set consists of two components containing $\Ctrop([\final, \lambda_i - \epsilon[)$ and $\Ctrop(]\lambda_i + \epsilon, \initial])$, respectively. This provides a contradiction.

From~\eqref{eq:breakpoint_neighborhood} and the strict monotonicity of the map $\Ctrop$, we know that the intersection of $\Ctrop([\final, \initial])$ with $\Hcal$ reduces to $\Ctrop([\lambda_i - \epsilon, \lambda_i + 2\epsilon])$.  Consequently, we have
\[
\Tcal \cap \Hcal \subset \Bcal_\infty(\Ctrop(\lambda_i); 3 \epsilon) \, .
\]
Summing up, for all $0 < i < p$, there is at least one index $j$ such that the ball $\Bcal_\infty(\Ctrop(\lambda_i); 3\epsilon)$ contains $z^j$.  This even holds for $i = 0$ or $i = p$, by the choices of $z^0$ and $z^p$.  As $\epsilon < \epsilon_0$, all these $p+1$ balls are pairwise disjoint.  We finally conclude that $p \geq \gamma$. 
\end{proof}

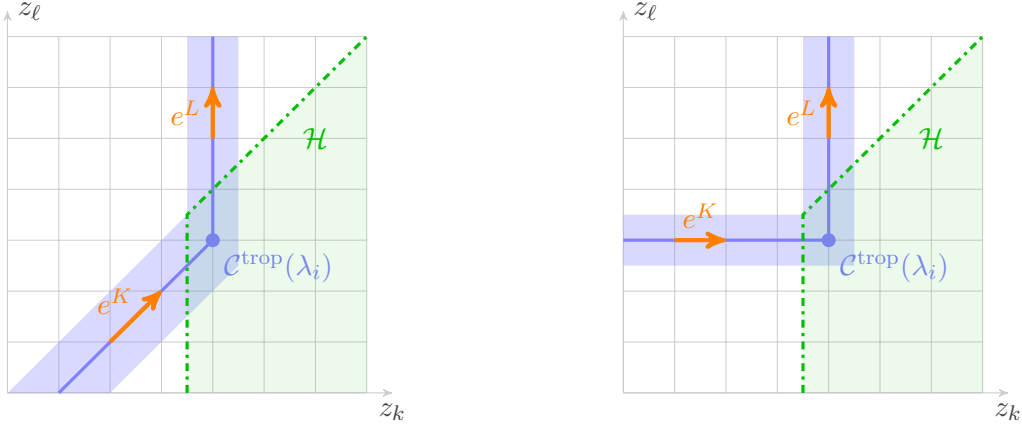
\begin{figure}
\begin{center}
\begin{tikzpicture}[scale=0.9,>=stealth',scale=0.75,  curve/.style={blue!50, very thick},tube/.style={color=lightgray, fill opacity = 0.7},vector/.style={color=orange, ultra thick},hs/.style={green!75!black,fill opacity=0.075},hp/.style={green!75!black,very thick, dashdotted}]

\begin{scope}
     \draw[help lines, gray!40!] (0,0) grid (7,7); 
     \draw[gray!40!, ->] (0,0) -- (7.5,0);
     \draw[gray!40!, ->] (0,0) -- (0,7.5);
     \node[anchor = north,darkgray] at (7.5,0) {$z_k$};
     \node[anchor = west,darkgray] at (0,7.5) {$z_\ell$};

\draw[curve] (1,0) -- (4,3) -- (4,7);
\filldraw[curve] (4,3) circle (3pt) node[below right] {$\Ctrop(\lambda_i)$};
\fill[curve,fill opacity=0.3] (2,0) -- (4.5,2.5) -- (4.5,7) -- (3.5,7) -- (3.5,3.5) -- (0,0)  -- cycle;

\draw[hp] (3.5,0) -- (3.5,3.5) -- (7,7);
\fill[hs] (3.5,0) -- (3.5,3.5) -- (7,7) -- (7,0) -- cycle;
\node[green!75!black] at (6,5) {$\Hcal$};

\draw[vector,->] (2,1) -- node[above left=-1ex] {$e^K$} (3,2);
\draw[vector,->] (4,5) -- node[left] {$e^L$} (4,6);
\end{scope}

\begin{scope}[shift={(12,0)}]
     \draw[help lines, gray!40!] (0,0) grid (7,7); 
     \draw[gray!40!, ->] (0,0) -- (7.5,0);
     \draw[gray!40!, ->] (0,0) -- (0,7.5);
     \node[anchor = north,darkgray] at (7.5,0) {$z_k$};
     \node[anchor = west,darkgray] at (0,7.5) {$z_\ell$};

\draw[curve] (0,3) -- (4,3) -- (4,7);
\filldraw[curve] (4,3) circle (3pt) node[below right] {$\Ctrop(\lambda_i)$};
\fill[curve,fill opacity=0.3] (0,2.5) -- (4.5,2.5) -- (4.5,7) -- (3.5,7) -- (3.5,3.5) -- (0,3.5)  -- cycle;

\draw[hp] (3.5,0) -- (3.5,3.5) -- (7,7);
\fill[hs] (3.5,0) -- (3.5,3.5) -- (7,7) -- (7,0) -- cycle;
\node[green!75!black] at (6,5) {$\Hcal$};

\draw[vector,->] (1,3) -- node[above] {$e^K$} (2,3);
\draw[vector,->] (4,5) -- node[left] {$e^L$} (4,6);
\end{scope}
\end{tikzpicture}
\end{center}
\caption{The tropical central path (in blue) and its tubular neighborhood $\Tcal$ (in light blue) near a breakpoint $\Ctrop(\lambda_i)$, projected onto the plane $(z_k, z_\ell)$.  The tropical halfspace $\Hcal$ defined in~\eqref{eq:transversal_halfspace} is shown in green.  The cases $\ell \in K$ and $\ell \not \in K$ are depicted left and right, respectively.}\label{fig:breakpoint_neighborhood}
\end{figure}

We are now ready to establish a general lower bound on the number of iterations performed by the class of log-barrier interior point methods described in Section~\ref{sec:classical_central_path}. The result is stated in terms of polygonal curves contained in the wide neighborhood $\Ncal^{-\infty}_{\theta,t}$ of the central path, since such curves are the trajectories followed by the log-barrier interior point methods.  Our proof combines Theorem~\ref{th:uniform_convergence_central_path} with Proposition~\ref{prop:piecewise_approximation}.  To this end, we pick a real number $t_1 \geq t_0$ such that
\begin{equation}
\log_t \Bigl(\frac{2N}{1-\theta}\Bigr) + \delta(t) < \epsilon_0 \label{eq:threshold}
\end{equation}
holds for all $t > t_1$.
Notice that $t_1$ depends on $t_0$, $\epsilon_0$ and, in particular, on the precision parameter $\theta$.
Hilbert's projective metric $\hilbert$ plays a role through the definition of $\delta(t)$ in \eqref{eq:delta_t}.
Recall that $\mgap(\cdot)$ measures the duality gap.

\begin{theorem}\label{th:tropical_lower_bound}
Suppose that $0 < \theta < 1$ and $t > t_1$.  Then, every polygonal curve $[z^0, z^1] \cup [z^1, z^2] \cup \dots \cup [z^{p-1}, z^p]$ contained in the neighborhood $\Ncal^{-\infty}_{\theta,t}$ satisfies
\[
p \geq \gamma\bigl([\log_t\mgap(z^0), \log_t \mgap(z^p)]\bigr) \, .
\]
\end{theorem}
\begin{proof}
We first assume that $\mgap(z^0) \leq \mgap(z^i) \leq \mgap(z^p)$ for all $i$.  Consider $z \in [z^i, z^{i+1}]$ for some $i$ with $0 \leq i < p$.  By Proposition~\ref{prop:mgap_affine} we have $\mgap(z^0) \leq \mgap(z) \leq \mgap(z^p)$.  It follows that $z \in \Ncal^{-\infty}_{\theta,t}(\mu)$ for some $\mu$ with $\mgap(z^0) \leq \mu \leq \mgap(z^p)$.

We define $\Scal$ as the union of the tropical segments $\tsegm(\log_t z^i, \log_t z^{i+1})$ for $0 \leq i < p$, and set $\final \coloneqq \log_t \mgap(z^0)$ and $\initial \coloneqq \log_t \mgap(z^p)$. By Theorem~\ref{th:uniform_convergence_central_path} and Lemma~\ref{lemma:uniform_convergence_segment}, we have:
\[
d_\infty\bigl(\Scal,\Ctrop([\final, \initial])\bigr) \leq \log_t\Bigl(\frac{2N}{1-\theta}\Bigr) + \delta(t) < \epsilon_0
\]
since $t>t_1$.
By choosing $\epsilon$ as $d_\infty\bigl(\Scal,\Ctrop([\final, \initial])\bigr)$, Theorem~\ref{th:uniform_convergence_central_path} ensures that $z^0 \in \Bcal_\infty(\Ctrop(\final); \epsilon)$ and $z^p \in \Bcal_\infty(\Ctrop(\initial); \epsilon)$. Therefore, we can apply Proposition~\ref{prop:piecewise_approximation}, which yields the claim in this special case.

We need to deal with the general case.  Let $j$ and $k$ be indices such that $\mgap(z^j)$ and $\mgap(z^k)$ are the minimal and maximal values among the $\mgap(z^i)$, respectively.  If $j = k$ then all duality measures $\mgap(z^i)$ agree, and $\gamma\bigl([\log_t\mgap(z^0), \log_t \mgap(z^p)]\bigr) = 1$.  If $j < k$, we can apply the argument for the special case to the subsequence $[z^j, z^{j+1}], \dots, [z^{k-1}, z^k]$.  This way we arrive at
\[
p \geq k-j \geq \gamma\bigl([\log_t\mgap(z^j), \log_t \mgap(z^k)]\bigr) \geq \gamma\bigl([\log_t\mgap(z^0), \log_t \mgap(z^p)]\bigr) \, ,
\]
where the latter inequality comes from the fact that $[\log_t\mgap(z^0), \log_t \mgap(z^p)]$ is contained in $[\log_t\mgap(z^j), \log_t \mgap(z^k)]$. The remaining case $j > k$ is similar.
\end{proof}

\subsection{An exponential lower bound on the number of iterations for our main example}\label{subsec:cex_complexity}

Let us consider the linear programs $\realCEXslack{r}{t}$ and $\dualRealCEXslack{r}{t}$.  A direct inspection reveals that choosing $t_0 = 0$ is sufficient to meet the requirements of Lemma~\ref{lemma:archimedean}. 

We focus on the section $\Ctrop([0,2])$ of the associated tropical central path, i.e., we consider $\final=0$ and $\initial=2$.  As explained in Section~\ref{subsec:cex_tropical_central_path} and illustrated in Figure~\ref{fig:x_central_path}, the projection of $\Ctrop\bigl([0,2]\bigr)$ onto the plane $(x_{2r-1}, x_{2r})$ consists of $2^{r-1}$ ordinary segments, which alternate their directions.  This projection to two dimensions cannot be expressed as a concatenation of less than $2^{r-1}$ tropical segments in the plane (see Figure~\ref{fig:tropical_segments}).  Therefore, also the tropical central path cannot be written as the union of fewer tropical segments in any higher dimensional space.  With our notation from Section~\ref{subsec:piecewise_approximation} this means that
\[
\gamma\bigl([0,2]\bigr) \geq 2^{r-1} \, .
\]

Moreover, it can be verified that the minimal $d_\infty$-distance between any two vertices in $\Ctrop$ equals $1/2^{r-2}$.  Therefore, choosing $\epsilon_0 = 1/(3 \cdot 2^{r-1})$ is good enough for being able to apply Proposition~\ref{prop:piecewise_approximation}.  It remains to find a sufficiently large number $t_1$ such that~\eqref{eq:threshold} holds for all $t > t_1$.  Every non-null coefficient in the constraint matrix of $\puiseuxCEX{r}$ is a monomial of degree in $\frac{1}{2^{r-1}} \Z$, and thus we may apply Theorem~\ref{th:polyhedron_metric_estimate}, where $\eta_0 \geq 1/2^{r-1}$.  As a consequence, if $t \geq ((2N)!)^{2^{r-1}}$, we have
\[
\hilbert(\log_t \bm \Fcal(t), \val(\bm \Fcal)) \leq \log_t\bigl((2N+1)^2 ((2N)!)^4\bigr) \, .
\]
Recall that $N = 5r-1$ is the total number of variables (including slacks).  Now Theorem~\ref{th:tropical_lower_bound} specializes to the following result.
\begin{theorem}\label{th:exp_lower_bound}
Let $0 < \theta < 1$, and suppose that 
\begin{equation}
t > \biggl(\max\Bigl((10r-2)!, \frac{\bigl((10r - 1)!\bigr)^{24}}{(1-\theta)^3}\Bigr)\biggr)^{2^{r-1}} \, . \label{eq:threshold_t}
\end{equation} 
Then, every polygonal curve $[z^0, z^1] \cup [z^1, z^2] \cup \dots \cup [z^{p-1}, z^p]$ contained in the neighborhood $\Ncal^{-\infty}_{\theta,t}$ of the primal-dual central path of $\realCEXslack{r}{t}$, with $\bar{\mu}(z^0) \leq 1$ and $\bar{\mu}(z^p) \geq t^2$, contains at least $2^{r-1}$ segments.
\end{theorem}

Taking into account the discussion at the end of Section~\ref{sec:classical_central_path}, we may restate Theorem~\ref{th:exp_lower_bound} in terms of the complexity of interior point methods and prove Theorem~\ref{thm:complexity:intro}.
\begin{corollary}\label{cor:exp_lower_bound}
Let $0 < \theta < 1$, and suppose that $t$ satisfies~\eqref{eq:threshold_t}. Then, any log-barrier interior point method which describes a trajectory contained in the neighborhood $\Ncal^{-\infty}_{\theta,t}$ of the primal-dual central path of $\realCEXslack{r}{t}$, needs to perform at least $2^{r-1}$ iterations to reduce the duality measure from $t^2$ to $1$.
\end{corollary}

\begin{remark}
Corollary~\ref{cor:exp_lower_bound} requires the size $\theta$ of the neighborhood to be fixed independently of the parameter $t$. While this requirement can be relaxed slightly (the lower bound holds as soon as $\log(1-\theta) = o(\log t)$), we point out that it is met by the interior point methods discussed in Section~\ref{sec:classical_central_path}. For instance, for predictor-corrector methods, the radius of the outer neighborhood is usually set to $\theta = 1/2$. Although this setting can be refined, one can show that the proof of the convergence requires $\theta$ to be chosen less than~$4/5$ (see~\cite[Exercise~5.6]{Wright}).
\end{remark}

\begin{remark}
It would be interesting to test standard linear programming solvers on the family $\realCEXslack{r}{t}$. It is obvious that solvers computing with bounded precision numbers are unable to deal with coefficients as large as in our example, in the light of the condition~\eqref{eq:threshold_t}. This already rules out most of the standard solvers. It would be worthwhile to explore how, \eg{}, SDPA-GMP can be used to deal with such input~\cite{SDPA-GMP}. That solver relies on the floating point numbers with arbitrary precision mantissa provided by the GMP library.
\end{remark}

\section{Combinatorial Experiments}\label{sec:combinatorics}
\noindent
We now want to give some hints to the combinatorial properties of the feasible region of the Puiseux linear program $\puiseuxCEX{r}$, which we denote as $\bm \Rcal_r$.
These are based on experiments for the first few values of $r$, which have been performed with \polymake \cite{DMV:polymake}.
Notice that since version 3.0 \polymake offers linear programming and convex hull computations over the field of Puiseux fractions with rational coefficients \cite{JoswigLohoLorenzSchroeter:2016}, and the coefficients of $\puiseuxCEX{r}$ lie in this subfield.
Notice that this is entirely independent of the metric analysis which was necessary for our main results.
Throughout we assume that $r\geq 1$.

By construction $\bm \Rcal_r$ is a convex polyhedron in the non-negative orthant of $\K^{2r}$ which contains interior points, which means that it is full-dimensional.
Moreover, it is easy to check that the exterior normal vectors of the defining inequalities positively span the entire space, hence $\bm \Rcal_r$ is bounded, i.e., a polytope over Puiseux series.
None of these $3r+1$ inequalities is redundant, i.e., each inequality defines a facet.
For instance, $\Rcal_1$ is a quadrangle.
However, the polytope $\bm \Rcal_r$ is not \emph{simple} for $r\geq 2$, i.e., there are vertices which are contained in more than $2r$ facets.
All these non-simple vertices lie in the optimal face, which is given by $x_1=0$.
By modifying the inequality $x_{2r}\geq 0$ to $x_{2r}\geq \epsilon$ for sufficiently small $\epsilon>0$ we obtain a simple polytope $\bm \Rcal_r^\epsilon$ as the feasible region of the \emph{perturbed linear program}
\begin{equation}\label{eq:counter:epsilon}
\begin{array}{r@{\quad}l}
\text{minimize} & x_1 \\[\jot]
\text{subject to} & 
x_1 \leq t^2 \\[\jot]
& x_2 \leq t \\[\jot]
& x_{2j+1} \leq t \, x_{2j-1} \, , \; x_{2j+1} \leq t \, x_{2j} \tikzmark{} \\[\jot]
& x_{2j+2} \leq t^{1-1/2^j} (x_{2j-1} + x_{2j}) \tikzmark{} \\[\jot]
& x_{2r-1} \geq 0 \, , \; x_{2r} \geq \epsilon
\end{array}\tag*{$\perturbedPuiseuxCEX{r}{\epsilon}$}
\insertbigbrace{$1 \leq j < r$}
\end{equation}
For the remainder of this section we refer to our original construction $\puiseuxCEX{r}$ and its feasible region as the \emph{unperturbed case}.
The unperturbed Puiseux polytope $\puiseuxCEX{r}$ can be seen as the limit of the perturbed Puisuex polytopes $\perturbedPuiseuxCEX{r}{\epsilon}$ when $\epsilon$ goes to zero.
See Figure~\ref{fig:schlegel} for a visualization of~$\bm \Rcal_2^\epsilon$.

The two facets $x_1=t^2$ and $x_3=tx_2$ play a special role. 
It can be verified that they do not share any vertices, neither in the perturbed nor in the unperturbed case.
In the unperturbed $r=2$ case these two facets cover all the vertices except for one, while four and $16$ are uncovered for $r=3$ and $r=4$, respectively.

Since the perturbation is only very slight it follows that the dual graphs of the perturbed and the unperturbed Puiseux polytope are the same.
For $2\leq r\leq 6$ we found that this is a complete graph minus one edge, which corresponds to the special pair of disjoint facets mentioned above.
This should be compared with the following constructions.
Let $D$ be any dual to $2$-neighborly polytope; any two facets of $D$ share a common \emph{ridge}, i.e., a face of codimension $2$.
Now pick any ridge and truncate it to produce a new polytope $D'$.
The dual graph of $D'$ is a complete graph minus one edge.
Note that the polytope $D'$ may have very many vertices as it is still very close to a polytope which is $2$-neighborly.

In our experiments, for all $r\leq 6$, the primal graph of $\bm \Rcal_r^\epsilon$ has diameter $r+1=(3r+1)-2r$, which is precisely the Hirsch bound.
This is in stark contrast with the unperturbed case in which the diameter equals $3$, for $2\leq r\leq 6$.

It is an interesting question which values for $\epsilon$ are small enough.
Our experiments suggest that $\epsilon=t^{-1}$ works for all $r \geq 2$.
Employing generalized Puiseux series with valuations of higher rank offers an alternative approach, which will always work:
we may introduce a second large infinitesimal $s\gg t$ and set $\epsilon=s^{-1}$.

\begin{figure}[ht]
  \centering

\begin{tikzpicture}[x  = {(0.998147445743233cm,-0.00742310781621397cm)},
                    y  = {(-0.0257708007775884cm,-0.950677647502943cm)},
                    z  = {(-0.0551139037225437cm,0.310091773526182cm)},
                    scale = 1.5,
                    color = {lightgray}]

  \footnotesize

  \definecolor{pointcolor}{rgb}{ 1,0,0 }
  \tikzstyle{pointstyle} = [fill=pointcolor]

  \coordinate (v0) at (0, 2, 2.82843);
  \coordinate (v1) at (1.79191, 1.1566, 3.57968);
  \coordinate (v2) at (2.75744, 1.1566, 4.94514);
  \coordinate (v3) at (0, 2, 0.5);
  \coordinate (v4) at (0, 0.353553, 0.5);
  \coordinate (v5) at (4, 0, 5.65685);
  \coordinate (v6) at (4, 2, 8.48528);
  \coordinate (v7) at (1.79191, 1.1566, 1.09014);
  \coordinate (v8) at (2.75744, 1.1566, 1.09014);
  \coordinate (v9) at (4, 0, 0.5);
  \coordinate (v10) at (4, 2, 0.5);
  \coordinate (v11) at (0.299698, 0.193441, 0.598702);
  \coordinate (v12) at (0.353553, 0, 0.5);

  \definecolor{linecolor_interior}{rgb}{ 0.4667,0.9255,0.6196 }
  \definecolor{linecolor_exterior}{rgb}{ 0,0,0 }
  \tikzstyle{linestyle_interior} = [color=blue!50, thick]
  \tikzstyle{linestyle_interior_front} = [preaction={draw=white, line width=3pt}, color=blue!50, thick]
  \tikzstyle{linestyle_exterior} = [color=linecolor_exterior, very thick]
  \tikzstyle{linestyle_exterior_front} = [preaction={draw=white, line width=5pt}, color=linecolor_exterior, very thick]

  \draw[linestyle_exterior] (v6) -- (v0);
  \draw[linestyle_interior] (v5) -- (v2);
  \draw[linestyle_interior] (v1) -- (v0);
  \draw[linestyle_interior] (v2) -- (v1);

  \node at (v0) [text=black, inner sep=0.5pt, above left, draw=none, align=left] {$(0, t, 0, t^{3/2})$};

  \draw[linestyle_exterior] (v12) -- (v5);
  \draw[linestyle_interior] (v6) -- (v2);
  \draw[linestyle_interior_front] (v7) -- (v1);
  \draw[linestyle_interior_front] (v7) -- (v3);
  \draw[linestyle_interior_front] (v8) -- (v2);

  \node at (v2) [rotate=-45, text=black, inner sep=3pt, left, draw=none, align=left, fill=white] {$(t^2, t, t^2, t^{5/2} + t^{3/2})$};

  \draw[linestyle_interior] (v8) -- (v7);
  \draw[linestyle_exterior] (v9) -- (v5);

  \node at (v3) [text=black, inner sep=0.5pt, below left, draw=none, align=left] {$(0, t, 0, t^{-1})$};

  \node at (v6) [text=black, inner sep=1.5pt, right, draw=none, align=left, fill=white] {$(t^2, t, 0, t^{5/2} + t^{3/2})$};

  \draw[linestyle_exterior] (v10) -- (v6);
  \draw[linestyle_exterior] (v6) -- (v5);
  \draw[linestyle_interior] (v10) -- (v8);

  \node at (v8) [text=black, inner sep=2pt, below, draw=none, align=left, fill=white]{$(t^2, t, t^2, t^{-1})$};

  \draw[linestyle_exterior] (v10) -- (v9);

  \node at (v10) [text=black, inner sep=0.5pt, below right, draw=none, align=left] {$(t^2, t, 0, t^{-1})$};

  \draw[linestyle_interior] (v11) -- (v1);

  \draw[linestyle_interior] (v11) -- (v4);
  \draw[linestyle_interior_front] (v11) -- (v7);

  \node at (v7) [text=black, inner sep=2pt, below left, draw=none, align=left, fill=white]{$(t, t, t^2, t^{-1})$};

  \node at (v4) [text=black, inner sep=1.5pt, left, draw=none, align=left] {$(0, t^{-3/2}, 0, t^{-1})$};

  \node at (v5) [text=black, inner sep=0.5pt, above right, draw=none, align=left] {$(t^2, 0, 0, t^{5/2})$};

  \node at (v9) [text=black, inner sep=1.5pt, right, draw=none, align=left] {$(t^2, 0, 0, t^{-1})$};

  \draw[linestyle_interior] (v12) -- (v11);

  \node at (v12) [rotate=-45, text=black, inner sep=2pt, above left, draw=none, align=right] {$(t^{-3/2}, 0, 0, t^{-1})$};
  \node at (v11) [rotate=-45, text=black, inner sep=3pt, left, draw=none, align=left] {$(\frac{1}{2}t^{-3/2}, \frac{1}{2}t^{-3/2}, \frac{1}{2}t^{-1/2}, t^{-1})$};

  \draw[linestyle_interior_front] (v9) -- (v8);
  \draw[linestyle_exterior_front] (v4) -- (v3);
  \draw[linestyle_exterior] (v4) -- (v0);
  \draw[linestyle_exterior] (v12) -- (v4);
  \draw[linestyle_exterior] (v10) -- (v3);
  \draw[linestyle_exterior] (v3) -- (v0);
  \draw[linestyle_exterior_front] (v12) -- (v9);

  \node at (v1) [rotate=-45, text=black, inner sep=0.5pt, above left, draw=none, align=left, fill=white]{$(t, t, t^2, 2t^{3/2})$};

  \foreach \i in {0,1,...,12}{
    \fill[pointcolor] (v\i) circle (1 pt);
  }

\end{tikzpicture}

  \caption{Schlegel diagram of perturbed polytope $\bm\Rcal_2^\epsilon$ (for $\epsilon=t^{-1}$) projected onto the facet $x_3=0$}
  \label{fig:schlegel}
\end{figure}
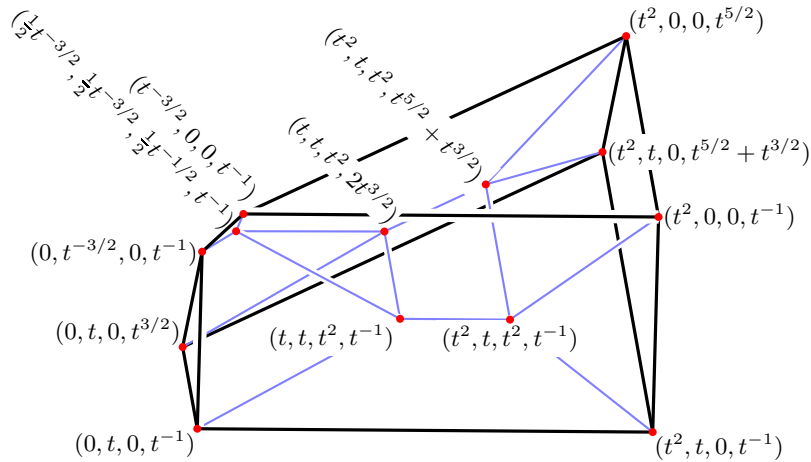

\section{Concluding Remarks}
\noindent
In the present work, we obtained a family of counter-examples showing that standard
polynomial-time interior points method exhibit a non strongly polynomial time
behavior. To do so, we considered nonarchimedean instances
with a degenerate tropical limit that we characterized
by combinatorial means. This strategy
is likely to be applicable to other problems
in computational complexity: tropicalization
generally permits to test the sensitivity of classical algorithms
to the bitlength of the input. 

Moreover, the present approach may also extend
to other interior point methods (e.g.~infeasible ones) or other barrier or penalty functions.
Indeed, as should be clear from ~\cite{alessandrini2013,longandwinding} 
what ``really'' matters is to work with a Hardy field
of functions definable in a o-minimal structure.
This allows for other fields than the
absolutely convergent generalized real Puiseux series
considered here.

The weak tropical angle $\angle^* UVW$ used in Proposition~\ref{prop:curvature} yields a bound on the total curvature of non-decreasing paths.
A similar approach allows one, more generally, to define a notion of tropical curvature for arbitrary paths.
This should also be compared with the notion of curvature for tropical hypersurfaces introduced in~\cite{bertrand}.
We leave this for future work.

\bibliographystyle{alpha}

\begin{thebibliography}{KMY89b}

\bibitem[ABG98]{ABG96}
M.~Akian, R.B. Bapat, and S.~Gaubert.
\newblock Asymptotics of the {P}erron eigenvalue and eigenvector using max
  algebra.
\newblock {\em C. R. Acad. Sci. Paris.}, 327, S\'erie I:927--932, 1998.

\bibitem[ABGJ14]{longandwinding}
X.~Allamigeon, P.~Benchimol, S.~Gaubert, and M.~Joswig.
\newblock Long and winding central paths.
\newblock E-print \arxiv{1405.4161}, 2014.

\bibitem[ABGJ15]{tropical+simplex}
X.~Allamigeon, P.~Benchimol, S.~Gaubert, and M.~Joswig.
\newblock Tropicalizing the simplex algorithm.
\newblock {\em SIAM J. Discrete Math.}, 29(2):751--795, 2015.
\newblock E-print \arxiv{1308.0454}.

\bibitem[AGS16]{tropical_spectrahedra}
X.~Allamigeon, S.~Gaubert, and M.~Skomra.
\newblock Tropical spectrahedra.
\newblock E-print \arxiv{1610.06746}, 2016.

\bibitem[Ale13]{alessandrini2013}
D.~Alessandrini.
\newblock Logarithmic limit sets of real semi-algebraic sets.
\newblock {\em Adv. Geom}, 13:155--190, 2013.

\bibitem[Ans91]{anstreicher1991performance}
K.M. Anstreicher.
\newblock On the performance of {K}armarkar's algorithm over a sequence of
  iterations.
\newblock {\em SIAM Journal on Optimization}, 1(1):22--29, 1991.

\bibitem[AR89]{nonsmoothpaths}
A.~D. Alexandrov and Yu.~G. Reshetnyak.
\newblock {\em General theory of irregular curves}.
\newblock Kluwer, 1989.

\bibitem[BdMR13]{bertrand}
B.~Bertrand, L.~L\'opez de~Medrano, and J.-J. Risler.
\newblock On the total curvature of tropical hypersurfaces.
\newblock In {\em Algebraic and combinatorial aspects of tropical geometry},
  number 589 in Contemp. Math., pages 21--43. Amer. Math. Soc., Providence, RI,
  2013.

\bibitem[Ber71]{bergman1971logarithmic}
G.M. Bergman.
\newblock The logarithmic limit-set of an algebraic variety.
\newblock {\em Transactions of the American Mathematical Society},
  157:459--469, 1971.

\bibitem[BH04]{BriecHorvath04}
W.~Briec and C.~Horvath.
\newblock $\mathbb{B}$-convexity.
\newblock {\em Optimization}, 53:103--127, 2004.

\bibitem[BL89]{BayerLagarias89a}
D.~A. Bayer and J.~C. Lagarias.
\newblock The nonlinear geometry of linear programming. {I}. {A}ffine and
  projective scaling trajectories.
\newblock {\em Trans. Amer. Math. Soc.}, 314(2):499--526, 1989.

\bibitem[BL97]{Bertsimas1997}
D.~Bertsimas and X.~Luo.
\newblock {On the worst case complexity of potential reduction algorithms for
  linear programming}.
\newblock {\em Mathematical Programming}, 77(2):321--333, May 1997.

\bibitem[BNRC08]{BezemNieuwenhuisRodriguez08}
M.~Bezem, R.~Nieuwenhuis, and E.~Rodr{\'{\i}}guez-Carbonell.
\newblock Exponential behaviour of the {B}utkovi\v c-{Z}immermann algorithm for
  solving two-sided linear systems in max-algebra.
\newblock {\em Discrete Appl. Math.}, 156(18):3506--3509, 2008.

\bibitem[BZ06]{Butkovic2006}
P.~Butkovi\v{c} and K.~Zimmermann.
\newblock {A strongly polynomial algorithm for solving two-sided linear systems
  in max-algebra}.
\newblock {\em Discrete Applied Mathematics}, 154(3):437--446, March 2006.

\bibitem[CGQ04]{cgq02}
G.~Cohen, S.~Gaubert, and J.P. Quadrat.
\newblock Duality and separation theorem in idempotent semimodules.
\newblock {\em Linear Algebra and Appl.}, 379:395--422, 2004.

\bibitem[DLSV12]{de2010central}
J.A. De~Loera, B.~Sturmfels, and C.~Vinzant.
\newblock The central curve in linear programming.
\newblock {\em Foundations of Computational Mathematics}, 12(4):509--540, 2012.

\bibitem[DMS05]{dedieu2005curvature}
J.-P. Dedieu, G.~Malajovich, and M.~Shub.
\newblock On the curvature of the central path of linear programming theory.
\newblock {\em Foundations of Computational Mathematics}, 5(2):145--171, 2005.

\bibitem[DS04]{develin2004}
M.~Develin and B.~Sturmfels.
\newblock Tropical convexity.
\newblock {\em Doc. Math.}, 9:1--27 (electronic), 2004.
\newblock correction: ibid., pp.\ 205--206.

\bibitem[DS05]{dedieu2005newton}
J.-P. Dedieu and M.~Shub.
\newblock Newton flow and interior point methods in linear programming.
\newblock {\em International Journal of Bifurcation and Chaos},
  15(03):827--839, 2005.

\bibitem[DTZ08]{deza2008polytopes}
A.~Deza, T.~Terlaky, and Y.~Zinchenko.
\newblock Polytopes and arrangements: diameter and curvature.
\newblock {\em Operations Research Letters}, 36(2):215--222, 2008.

\bibitem[DTZ09]{DTZ08}
A.~Deza, T.~Terlaky, and Y.~Zinchenko.
\newblock Central path curvature and iteration-complexity for redundant
  {K}lee-{M}inty cubes.
\newblock In {\em Advances in applied mathematics and global optimization},
  volume~17 of {\em Adv. Mech. Math.}, pages 223--256. Springer, New York,
  2009.

\bibitem[DY07]{DevelinYu07}
M.~Develin and J.~Yu.
\newblock Tropical polytopes and cellular resolutions.
\newblock {\em Experiment. Math.}, 16(3):277--291, 2007.

\bibitem[EKL06]{kapranov}
M.~Einsiedler, M.~Kapranov, and D.~Lind.
\newblock Non-{A}rchimedean amoebas and tropical varieties.
\newblock {\em J. Reine Angew. Math.}, 601:139--157, 2006.

\bibitem[GGK04]{CharlesGilbert2004}
J.C. Gilbert, C.C. Gonzaga, and E.~Karas.
\newblock {Examples of ill-behaved central paths in convex optimization}.
\newblock {\em Mathematical Programming}, 103(1):63--94, December 2004.

\bibitem[GJ00]{DMV:polymake}
E.~Gawrilow and M.~Joswig.
\newblock \polymake: a framework for analyzing convex polytopes.
\newblock In {\em Polytopes---combinatorics and computation (Oberwolfach,
  1997)}, volume~29 of {\em DMV Sem.}, pages 43--73. Birk\-h\"au\-ser, Basel,
  2000.

\bibitem[GK11]{GaubertKatz2011minimal}
S.~Gaubert and R.D. Katz.
\newblock Minimal half-spaces and external representation of tropical
  polyhedra.
\newblock {\em Journal of Algebraic Combinatorics}, 33(3):325--348, 2011.

\bibitem[HR15]{hardy}
G.H. Hardy and M.~Riesz.
\newblock {\em The general theory of {D}irichlet's series}.
\newblock Cambridge University Press, 1915.

\bibitem[IMS07]{itenberg2009tropical}
I.~Itenberg, G.~Mikhalkin, and E.~Shustin.
\newblock {\em Tropical algebraic geometry}, volume~35 of {\em Oberwolfach
  Seminars}.
\newblock Birkh\"auser Verlag, Basel, 2007.

\bibitem[JLLS16]{JoswigLohoLorenzSchroeter:2016}
M.~Joswig, G.~Loho, B.~Lorenz, and B.~Schr\"oter.
\newblock Linear programs and convex hulls over fields of {P}uiseux fractions.
\newblock In {\em Proceedings of MACIS 2015, Berlin, November 11--13, 2015.
  LNCS 9582}, pages 429--445. Springer, 2016.

\bibitem[JY94]{ji1994complexity}
J.~Ji and Y.~Ye.
\newblock A complexity analysis for interior-point algorithms based on
  {K}armarkar's potential function.
\newblock {\em SIAM Journal on Optimization}, 4(3):512--520, 1994.

\bibitem[Kar84]{karmarkar1984new}
N.~Karmarkar.
\newblock A new polynomial-time algorithm for linear programming.
\newblock {\em Combinatorica}, 4(4):373--395, 1984.

\bibitem[KMY89a]{kojima_short_step}
M.~Kojima, S.~Mizuno, and A.~Yoshise.
\newblock A polynomial-time algorithm for a class of linear complementarity
  problems.
\newblock {\em Math. Programming}, 44(1, (Ser. A)):1--26, 1989.

\bibitem[KMY89b]{kojima_long_step}
M.~Kojima, S.~Mizuno, and A.~Yoshise.
\newblock A primal-dual interior point algorithm for linear programming.
\newblock In {\em Progress in mathematical programming ({P}acific {G}rove,
  {CA}, 1987)}, pages 29--47. Springer, New York, 1989.

\bibitem[KOT13]{kakihara2013information}
S.~Kakihara, A.~Ohara, and T.~Tsuchiya.
\newblock Information geometry and interior-point algorithms in semidefinite
  programs and symmetric cone programs.
\newblock {\em Journal of Optimization Theory and Applications},
  157(3):749--780, 2013.

\bibitem[KT13]{KitaharaTsuchiya}
T.~Kitahara and T.~Tsuchiya.
\newblock A simple variant of the mizuno--todd--ye predictor-corrector
  algorithm and its objective-function-free complexity.
\newblock {\em SIAM Journal on Optimization}, 23(3):1890--1903, 2013.

\bibitem[KY91]{kaliski1991convergence}
J.A. Kaliski and Y.~Ye.
\newblock Convergence behavior of {K}armarkar's projective algorithm for
  solving a simple linear program.
\newblock {\em Operations research letters}, 10(7):389--393, 1991.

\bibitem[Lit07]{litvinov2007maslov}
G.L. Litvinov.
\newblock Maslov dequantization, idempotent and tropical mathematics: a brief
  introduction.
\newblock {\em Journal of Mathematical Sciences}, 140(3):426--444, 2007.

\bibitem[MA89]{AdlerMonteiro}
R.D.C. Monteiro and I.~Adler.
\newblock Interior path following primal-dual algorithms. part i: Linear
  programming.
\newblock {\em Mathematical Programming}, 44:27--41, 1989.

\bibitem[Mar10]{markwig2007field}
T.~Markwig.
\newblock A field of generalised {P}uiseux series for tropical geometry.
\newblock {\em Rend. Semin. Mat., Univ. Politec. Torino}, 68(1):79--92, 2010.

\bibitem[MMT98]{MegiddoMizunoTsuchiya}
N.~Megiddo, S.~Mizuno, and T.~Tsuchiya.
\newblock A modified layered-step interior-point algorithm for linear
  programming.
\newblock {\em Mathematical Programming}, 82(3):339--355, Aug 1998.

\bibitem[MS15]{TropicalBook}
D.~Maclagan and B.~Sturmfels.
\newblock {\em Introduction to Tropical Geometry}, volume 161 of {\em Graduate
  Texts in Math.}
\newblock American Math. Soc., 2015.

\bibitem[MT03]{MonteiroTsuchiya}
R.~D.~C. Monteiro and T.~Tsuchiya.
\newblock A variant of the vavasis--ye layered-step interior-point algorithm
  for linear programming.
\newblock {\em SIAM Journal on Optimization}, 13(4):1054--1079, 2003.

\bibitem[MTY93]{MizunoToddYe}
S.~Mizuno, M.~J. Todd, and Y.~Ye.
\newblock On adaptive-step primal-dual interior-point algorithms for linear
  programming.
\newblock {\em Mathematics of Operations Research}, 18(4):964--981, 1993.

\bibitem[Nak10]{SDPA-GMP}
M.~Nakata.
\newblock A numerical evaluation of highly accurate multiple-precision
  arithmetic version of semidefinite programming solver: {SDPA-GMP}, {-QD} and
  {-DD}.
\newblock In {\em IEEE International Symposium on Computer-Aided Control System
  Design (CACSD)}, 2010.

\bibitem[Pow93]{powell1993number}
M.J.D. Powell.
\newblock On the number of iterations of {K}armarkar's algorithm for linear
  programming.
\newblock {\em Mathematical Programming}, 62(1-3):153--197, 1993.

\bibitem[PT14]{PT14}
A.~Papadopoulos and M.~Troyanov.
\newblock Weak {M}inkowski spaces.
\newblock In {\em Handbook of Hilbert Geometry}, pages 11--32. Eur. Math. Soc.
  Publishing House, 2014.
\newblock To appear.

\bibitem[RGST05]{richter2005first}
J.~Richter-Gebert, B.~Sturmfels, and T.~Theobald.
\newblock First steps in tropical geometry.
\newblock In {\em Idempotent mathematics and mathematical physics}, volume 377
  of {\em Contemp. Math.}, pages 289--317. Amer. Math. Soc., Providence, RI,
  2005.

\bibitem[Sma00]{smale}
S.~Smale.
\newblock Mathematical problems for the next century.
\newblock In {\em Mathematics: frontiers and perspectives}, pages 271--294.
  Amer. Math. Soc., Providence, RI, 2000.

\bibitem[SSZ91]{sonnevend1991complexity}
G.~Sonnevend, J.~Stoer, and G.~Zhao.
\newblock On the complexity of following the central path of linear programs by
  linear extrapolation {II}.
\newblock {\em Mathematical Programming}, 52(1-3):527--553, 1991.

\bibitem[TY96]{todd1996lower}
M.J. Todd and Y.~Ye.
\newblock A lower bound on the number of iterations of long-step primal-dual
  linear programming algorithms.
\newblock {\em Annals of Operations Research}, 62(1):233--252, 1996.

\bibitem[vdDS98]{Dries1998}
L.~van~den Dries and P.~Speissegger.
\newblock The real field with convergent generalized power series.
\newblock {\em Trans. Amer. Math. Soc.}, 350(11):4377--4421, 1998.

\bibitem[Vir01]{viro2001dequantization}
O.~Viro.
\newblock Dequantization of real algebraic geometry on logarithmic paper.
\newblock In {\em European Congress of Mathematics}, pages 135--146. Springer,
  2001.

\bibitem[VY96]{VavasisYe}
S.~A. Vavasis and Y.~Ye.
\newblock A primal-dual interior point method whose running time depends only
  on the constraint matrix.
\newblock {\em Mathematical Programming}, 74(1):79--120, Jul 1996.

\bibitem[Wri97]{Wright}
S.~Wright.
\newblock {\em Primal-Dual Interior-Point Methods}.
\newblock Society for Industrial and Applied Mathematics, 1997.

\bibitem[ZS93]{zhao1993estimating}
G.~Zhao and J.~Stoer.
\newblock Estimating the complexity of a class of path-following methods for
  solving linear programs by curvature integrals.
\newblock {\em Applied Mathematics and Optimization}, 27(1):85--103, 1993.

\end{thebibliography}

\end{document}